%% file: main.tex
\documentclass[a4paper]{article}
\usepackage{amsthm,amssymb,amsmath,enumerate,graphicx,epsf}
\usepackage{bbold}
\usepackage{mathtools,enumitem}
\usepackage{authblk}
\usepackage[percent]{overpic}
\usepackage{hyperref}
\RequirePackage{latexsym} \RequirePackage{amsthm}
\RequirePackage{amsmath}
\usepackage{authblk}
\usepackage{tikz}
\usepackage[dvipsnames]{xcolor}

\usepackage{enumerate}
\RequirePackage{amssymb} \RequirePackage{makeidx}
\usepackage{dsfont}
\usepackage{mathrsfs}
\usepackage[noabbrev,capitalise,nameinlink]{cleveref}

\newcommand{\COLORON}{1}
\newcommand{\NOTESON}{0}
\newcommand{\Debug}{0} 

\input{defs}

\newcommand{\diam}{\mathrm{diam}}

\newcommand{\ncm}{near-component}

\DeclareMathOperator{\dist}{dist}

\begin{document}

\title{Excluding $K_{2,t}$ as a fat minor}

\author[1]{Sandra Albrechtsen\thanks{Supported by the Alexander von Humboldt Foundation in the framework of the Alexander von Humboldt Professorship of Daniel Král' endowed by the Federal Ministry of Education and Research.}}
\affil[1]{{Institute of Mathematics}\\ {Leipzig University}\\ {Augustusplatz 10}\\ {04109 Leipzig}\\ {Germany}}
\author[2]{Marc Distel\thanks{Supported by Australian Government Research Training Program Scholarship.}}
\affil[2]{{School of Mathematics}\\ {Monash University}\\ {Melbourne}\\ {Australia}}
\author[3]{Agelos Georgakopoulos\thanks{Supported by EPSRC grant  EP/V009044/1.}}
\affil[3]{{Mathematics Institute}\\ {University of Warwick}\\  {CV4 7AL, UK}}

\date{October 17, 2025}
\maketitle

\begin{abstract}
We prove that for every $t \in \N$, the graph $K_{2,t}$ satisfies the fat minor conjecture of Georgakopoulos and Papasoglu:  \fe\ $K\in \N$ there exist $M,A\in \N$ such that every graph with no $K$-fat $K_{2,t}$ minor is $(M,A)$-quasi-isometric to a graph with no $K_{2,t}$ minor. We use this to obtain an efficient algorithm for approximating the minimal multiplicative distortion of any embedding of a finite graph into a $K_{2,t}$-minor-free graph, answering a question of Chepoi, Dragan, Newman, Rabinovich, and Vax\`es from 2012.
\end{abstract}

{\bf{Keywords:} } coarse graph theory, quasi-isometry, asymptotic minor,\\ approximation algorithm.\\

{\bf{MSC 2020 Classification:}} 05C83, 05C10, 68R12, 05C63, 51F30.

\maketitle

\section{Introduction}
Coarse graph theory is a rapidly developing new area that studies graphs from a geometric perspective, and conversely, transfers graph-theoretic results to metric spaces. The focus is on large-scale properties of the graphs and spaces involved, in particular on properties that are stable under quasi-isometries (defined in \Sr{sec QI}). A central notion of this area is that of a \defi{$K$-fat minor}, a geometric analogue of the classical notion of graph minor whereby branch sets are required to be at distance at least some distance $K$ from each other, and the edges connecting them are replaced by long paths, also at distance  $K$ from each other, and from their non-incident branch sets; see \Sr{sec FM} for details. We say that a graph $J$ is an \defi{asymptotic minor} of a graph $G$, if $J$ is a $K$-fat minor of $G$ \fe\ $K\in \N$. For any fixed $J$, this property is easily seen to be invariant under quasi-isometry on $G$ (\cite[Observation~2.4]{GeoPapMin}). 

Much of the impetus of coarse graph theory is due to the following conjecture of \cite{GeoPapMin}:

\begin{conjecture}[\cite{GeoPapMin}] \label{conj fat min}
For every finite graph $J$ and every $K\in \N$ there exist $M,A\in \N$  such that every graph with no $K$-fat $J$~minor is $(M,A)$-quasi-isometric to a graph with no $J$~minor.
\end{conjecture}
In other words, the conjecture asks whether every graph (family) forbidding $J$ as an asymptotic minor is (uniformly) quasi-isometric with a graph (family) forbidding $J$ as a minor. This was a natural conjecture to make, as the converse is easily seen to be true. However, \Cnr{conj fat min} was disproven by Davies, Hickingbotham, Illingworth and McCarty \cite{DHIM}. In a companion paper \cite{ADGSmallCounterexamples} we will provide much smaller counterexamples; in particular, we will prove that it is false for $J=K_t, t\geq 6$, and for $K_{s,t}, s,t\geq 4$. 
Recently, Albrechtsen and Davies~\cite{ADWeakCounterex} also disproved a weaker version of Conjecture~\ref{conj fat min}, stated in \cite{DHIM}, postulating a quasi-isometry to a graph forbidding some possibly much larger graph $J'$ as a minor.

This negative answer to \Cnr{conj fat min} fuels the interest in the broader quest, already initiated by Bonamy et al.\ \cite{BBEGLPS}, to understand graphs (or graph families) with a forbidden asymptotic minor. A substantial aspect of this quest, motivating the current paper, is to understand the limits of the validity of the conjecture. Several positive results have been obtained so far: \Cnr{conj fat min} is true e.g.\ for $J = K_3$ (more generally, for any cycle $J$) \cite{GeoPapMin}, for $J = K_{1,t}$ \cite{GeoPapMin,NgScSeAsyII}, for $J = K^-_4$ \cite{FujPapCoa, AJKW}, $J=K_{2,3}$ \cite{CDNRV,FujPapCoa},  and $J = K_4$ \cite{AJKW}. An important open question, due to its connection with induced minors,  is whether \Cnr{conj fat min} is true for $K=2$.

Given the above results, a central outstanding case towards understanding which graphs satisfy Conjecture~\ref{conj fat min} is the case $J= K_{2,t}, t\geq 4$. This question is implicit in earlier work of Chepoi, Dragan, Newman, Rabinovich and Vax\`es \cite{CDNRV}, where a variant of the notion of fat minor is introduced. The aim of this paper is to settle this question in the affirmative; we prove

\begin{theorem} \label{thm:FatK2t}
    For every $t \in \N$ there exists a function $f: \N \rightarrow \N^2$ such that every graph with no $K$-fat $K_{2,t}$ minor is $f(K)$-quasi-isometric to a graph with no $K_{2,t}$ minor. 
\end{theorem}

We remark that this problem bears some similarity to the coarse Menger conjecture \cite{AHJKW,GeoPapMin}, which has been disproven even in a much weaker form \cite{NgScSeAsyIV}.

\medskip
Our proof is constructive, and we obtain the bound $(9t^{12}K+204t^9K, 1)$ on $f(K)$. In other words, the additive distortion we obtain is $1$, and the multiplicative distortion $O(K)$. From this it is easy to obtain a map of additive distortion~$0$ (and still with multiplicative distortion $O(K)$ \cite[Observation 2.2]{GeoPapMin}\footnote{The additive error can always be hidden inside the multiplicative factor, unless more than one vertex of $G$ is mapped to the same vertex of $H$. In this case, attach a star of size $|V(G)|$ to each vertex $h$ of $H$ (which does not create any $K_{2,t}$ minors), and for each vertex $v$ of $G$ previously mapped to $h$, map $v$ to a distinct leaf of the star attached to $h$.}).

Given a finite graph \G, let \defi{$\alpha_t(G)$}  denote the minimal multiplicative distortion of any embedding of \g into a $K_{2,t}$-minor-free graph. Chepoi et al.\ \cite{CDNRV} asked whether there is an efficient algorithm that approximates $\alpha_t(G)$ to a constant factor. Using the above remarks we answer this question in the affirmative: 

\begin{corollary} \label{cor:algo}
    For every $t\in \N$, there is a polynomial-time algorithm that given a finite graph \G, approximates $\alpha_t(G)$ up to a universal multiplicative constant. 
\end{corollary}

We prove this in \Sr{sec:algo}, where we offer some related open problems. 

\subsection{Other problems}
As mentioned above, \Tr{thm:FatK2t} becomes false if we replace 
$K_{2,t}$ by $K_{4,t}, t\geq 4$ (even in a weak form as in Question 1.2 below), but we do not know if it is true for $K_{3,t}, t\geq 3$. The case $J=K_{3,3}$ is particularly important, as it is closely related to the `coarse Kuratowski conjecture' of \cite{GeoPapMin}:

\begin{question}
    Are there functions $f: \N \rightarrow \N^2$ and $s: \N \rightarrow \N$ such that every graph with no $K$-fat $K_{3,t}$ minor is $f(K)$-quasi-isometric to a graph with no $K_{3,s(t)}$ minor? Can we choose $s(t) = t$?
\end{question}

Another question of \cite{GeoPapMin} is for which $J$ we can achieve $M=1$ in \Cnr{conj fat min}, and variants of this question are discussed by Nguyen, Scott and Seymour \cite{NgScSeAsyII,NgScSeCoa}. Settling this for $J=K_{2,t}$ would be interesting, but our proof does not provide evidence.

\subsection{Proof approach}
Like many results in the area, our proof of \Tr{thm:FatK2t} is achieved by decomposing the vertex set of the underlying graph \g into `bags', of bounded diameter, so that after collapsing each bag into a vertex, the resulting graph $H$ is quasi-isometric to \G. The standard technique is to achieve such a decomposition by first decomposing \g into its distance layers from a fixed `root' vertex, and place nearby vertices of a fixed layer, or a fixed number of consecutive layers, into a bag, see e.g.\ \cite[Theorem~3.1]{GeoPapMin}. Our decomposition is based on a rather intricate refinement of this technique, whereby the number of consecutive layers from which a bag is formed is not fixed but depends on the local structure. Once $H$ is constructed, one then needs a way to turn any $K_{2,t}$ minor of $H$ into a $K$-fat minor of \G; this is not straightforward, one of the difficulties being that bags are not necessarily connected. Thus our proof requires new ideas involving a new way of forming branch sets in \g out of bags in $H$ by using vertices from bags of lower layers. To ensure that distinct branch sets are $K$-far apart, we use a new `buffer zone' technique within each bag, i.e.\ a sequence of layers that can only be used to accommodate branch paths. A more detailed overview of our proof is given in \Sr{sec:overview}. 

\section{Preliminaries}

Graphs in this paper are allowed to be infinite, unless stated otherwise. 
We follow the basic graph-theoretic terminology of~\cite{Bibel}; 
in particular, $\N$ includes $0$, and we denote by \defi{$||G||$} the number of edges of a graph $G$. Note that if $P$ is a path, then $||P||$ is its length.
Moreover, a set $U$ of vertices in a graph $G$ is \defi{connected}, if the subgraph $G[U]$ it induces is connected. 

Given a graph $G$, we write \defi{$\cc(G)$} for the set of components of $G$.
Given a subgraph $Y$ of~$G$, the \defi{boundary} \defi{$\partial_G Y$} of $Y$ is the set of all vertices of $Y$ that send an edge to $G-Y$. The \defi{neighbourhood $N_G(Y)$} of $Y$ is the set of vertices of  $G-Y$ sending an edge to $Y$ (and therefore to $\partial_G Y$).

\subsection{Distances}

Let $G$ be a graph.
We write~\defi{$d_G(v, u)$} for the distance between two vertices~$v$ and~$u$ in~$G$. 
For two sets~$U$ and~$U'$ of vertices of~$G$, we write~\defi{$d_G(U, U')$} for the minimum distance of two elements of~$U$ and~$U'$, respectively.
If one of~$U$ or~$U'$ is just a singleton, then we omit the braces, writing $d_G(v, U') := d_G(\{v\}, U')$ for $v \in V(G)$.

Given a set~$U$ of vertices of~$G$, the \defi{ball (in~$G$) around~$U$ of radius $r \in \N$}, denoted by~\defi{$B_G(U, r)$}, is the set of all vertices in~$G$ of distance at most~$r$ from~$U$ in~$G$.
If~$U = \{v\}$ for some~$v \in V(G)$, then we again omit the braces, writing~$B_G(v, r)$ instead of $B_G(\{v\}, r)$. 

The \defi{diameter}~\defi{$\diam(G)$} of~$G$ is the smallest number~$k \in \N \cup \{\infty\}$ such that $d_G(u,v) \leq k$ for every two~$u,v \in V(G)$. If $G$ is empty, then we define its diameter to be~$0$.
We remark that if $G$ is disconnected but not the empty graph, then its diameter is $\infty$.
The \defi{diameter of a set $U\subseteq V(G)$ in $G$}, denoted by \defi{$\diam_G(U)$}, is the smallest number $k \in \N$ such that $d_G(u,v) \leq k$ for all $u,v \in U$ or $\infty$ if such a $k \in \N$ does not exist. 

If $Y$ is a subgraph of $G$, then we abbreviate $d_G(U,V(Y))$, $\diam_G(V(Y))$ and $B_G(V(Y),r)$ as $d_G(U,Y)$, $\diam_G(Y)$ and $B_G(Y,r)$, respectively.
\medskip

Let $G$ be a graph. We say that $U \subseteq V(G)$ is \defi{$K$-near-connected} for $K\in \N$, if for every $x,y\in U$, there is a sequence $x=x_0,x_1, \ldots, x_k=y$ of vertices in $U$ such that $d(x_i,x_{i+1})\leq K$ \fe\ $i<k$.
Such a sequence $P=x_0, \dots, x_k$ will be called an $K$-\defi{near path} from $x$ to $y$. 
A \defi{$K$-\ncm} of $U$ is a maximal subset of $U$ that is $K$-near-connected.

\subsection{Fat minors} \label{sec FM}

Let $J, G$ be (multi-)graphs.
A \defi{model} $(\cu,\ce)$ of $J$ in $G$ is a collection $\cu$ of disjoint, connected sets $U_x \subseteq V(G), x\in V(J)$, and a collection $\ce$ of internally disjoint $U_{x}$--$U_{y}$ paths $E_{e}$, one for each edge $e=xy$ of $J$, \st\ $E_{e}$ is disjoint from every $U_z$ with $z \neq x, y$.
The $U_x$ are the \defi{branch sets} and the $E_e$ are the \defi{branch paths} of the model.
A model $(\cu, \ce)$ of $J$ in $G$ is \defi{$K$-fat} for $K \in \N$ if $\dist_G(Y,Z) \geq K$ for every two distinct $Y,Z \in \cu \cup \ce$ unless $Y = E_e$ and $Z = U_x$ for some vertex $x \in V(J)$ incident to $e \in E(J)$, or vice versa.
We say that $J$ is a \defi{($K$-fat) minor} of $G$, if $G$ contains a ($K$-fat) model of $X$.
We remark that the $0$-fat minors of $G$ are precisely its minors. 

\begin{lemma} \label{lem:SubdivingEdgesOfFatMinor}
    Let $J, G$ be (multi-)graphs, and let $\dot{J}$ be the graph obtained from $J$ by subdividing each of its edges precisely once. If $J$ is a $3K$-fat minor of~$G$ for some $K \in \N$, then $\dot{J}$ is a $K$-fat minor of $G$.
\end{lemma}

This lemma is a variant of \cite[Lemma~5.3]{GeoPapMin}; we include a proof for convenience.

\begin{proof}
    Let $(\cu, \ce)$ be a $3K$-fat model of $J$ in $G$. We construct a $K$-fat model $(\cu', \ce')$ of $\dot{J}$ in $G$ as follows. For every $x \in V(J)$, we keep $U'_x := U_x$ as a branch set. For every edge $e = xy \in E(J)$, we let $u_e$ be the last vertex on $E_e$, as we move from $U_x$ to $U_y$ along $E_e$, such that $d_G(U_x, u_e) \leq K$, and we let $v_e$ be the first vertex after $u_e$ along $E_e$ such that $d_G(U_y, v_e) \leq K$. We let the branch set~$U'_{w_e}$ for the subdivision vertex of $\dot{J}$ on $e$ be the subpath of $E_e$ between $u_e$ and~$v_e$. For $z \in \{x,y\}$, we let $E'_{zw_e}$ be an $U'_z$--$U'_{w_e}$ path of length $K$. This completes the definition of $(\cu', \ce')$.

    As $(\cu,\ce)$ is $3K$-fat and $E'_{xw_e} \subseteq B_G(E_e,K)$ for all edges of $\dot{J}$, we have $d_G(E'_{xw_e}, E'_{yw_f})$ $\geq 3K-2K=K$ for all edges $xw_e \neq yw_f$ of $\dot{J}$, unless $e = f$, in which case we have $d_G(E'_{xw_e}, E'_{yw_e}) \geq d_G(U_x, U_y) - ||E'_{xw_e}|| - ||E'_{yw_e}|| = 3K-K-K = K$ by the choice of the branch paths of $\dot{J}$.
    Similarly and because $U'_{w_e} \subseteq E_e$ for all subdivision vertices of $\dot{J}$, we have $d_G(U'_x, U'_y) \geq 3K$ for all $x \neq y \in V(\dot{J})$, unless one of $x,y$ is a subdivision vertex $w_e$ on an edge $e$ of $J$ incident with the other, in which case we have $d_G(U'_x, U'_y) \geq K$ by the choice of the $U'_{w_e}$. 
    Hence, it remains to consider $x \in V(\dot{J})$ and $yw_e \in E(\dot{J})$. If $x$ is a subdivision vertex on an edge $f$ of $J$, then $d_G(U'_x, E'_{yw_e}) \geq d_G(E_f, E_e) \geq 3K-K=2K$. 
    Otherwise, $d_G(U'_x, E'_{yw_e}) \geq d_G(U_x, U_y)- ||E'_{yw_e}|| = 3K - K = 2K$, as desired.
\end{proof}

\subsection{Quasi-isometries and graph-partitions} \label{sec QI}

Let $G,H$ be graphs.
For $M \in \R_{\geq 1}$ and $A \in \R_{\geq 0}$, an \defi{$(M, A)$-quasi-isometry} from~$G$ to~$H$ is a map~$\varphi : V(G) \rightarrow V(H)$ such that
\begin{enumerate}[label=\rm{(Q\arabic*)}]
    \item \label{quasiisom:1} $M^{-1} \cdot d_G(u,v) - A \leq d_H(\varphi(u),\varphi(v)) \leq M\cdot d_G(u,v)+A$ for every $u,v \in V(G)$, and
    \item \label{quasiisom:2} for every $h \in V(H)$ there is $v \in V(G)$ such that $d_H(h,\varphi(v)) \leq A$.
\end{enumerate}

We say that a map ~$\varphi : V(G) \rightarrow V(H)$ has \defi{multiplicative distortion} $M$ (respectively, \defi{additive distortion} $A$) if it satisfies \ref{quasiisom:1} with $A=0$ (resp.\ $M=1$). 
\smallskip

A \defi{graph-partition of}~$G$ \defi{over}~$H$, or $H$-\defi{partition} for short, is a partition $\ch := (V_h : h \in V(H))$ of~$V(G)$ indexed by the nodes of~$H$ such that for every edge $uv \in E(G)$, if $u \in V_{g}$ and $v \in V_h$, then $g = h$ or $gh \in E(H)$.
(This notion generalizes tree-partitions.) 

We say that $\ch$ is \defi{honest}, if $V_h$ is non-empty for all $h \in V(H)$ and if for every edge $gh \in E(H)$ there exists an edge $uv \in V(G)$ such that $u \in V_g$ and $v \in V_h$. 
We say that $\ch$ is \defi{$R$-bounded}, if each $V_h$ has diameter at most~$R(K)$.

\begin{lemma} \label{lem part QI}
    Let $H, G$ be graphs, and let $\ch$ be an honest, $R$-bounded $H$-partition  of $G$ for some $R\in \R$. Then $G$ is $(R + 1, R/(R+1))$-quasi-isometric to~$H$. 
\end{lemma}
This is a special case of \cite[Lemma~3.9]{radialpathwidth}; we include a proof for convenience:
\begin{proof} 
    As the $V_h$ are pairwise disjoint and cover $V(G)$, there is for every $v \in V(G)$ a unique $h_v \in V(H)$ such that $v \in V_{h_v}$. We claim that $\varphi: V(G) \rightarrow V(H)$ with $\varphi(v) := h_v$ is the desired quasi-isometry from~$G$ to~$H$. Let us check that $\varphi$ satisfies both properties of the definition of quasi-isometry:

    \ref{quasiisom:2}: As the $V_h$ are non-empty, there is for every $h \in V(H)$ some $v \in V(G)$ such that $h = \varphi(v)$, and hence $h$ has distance $0$ from $\varphi(v)$.

    \ref{quasiisom:1}: Fix $u, v \in V(G)$. Since $w \in V_{\varphi(w)}$ for all $w \in V(H)$, every $u$--$w$ path~$P$ in~$G$ of length $\ell \in \N$ induces a $\varphi(u)$--$\varphi(w)$ walk in~$H$ of length at most~$\ell$ with vertex set $\{h \in V(H) \mid \exists p \in V(P) : p \in V_{h}\}$. 
    Hence, $d_H(\varphi(u), \varphi(v)) \leq d_G(u,v)$. 
    
    Conversely, every $\varphi(u)$--$\varphi(v)$ path in~$H$ of length~$\ell$ can be turned into a $u$--$v$ walk in~$G$ of length at most $\ell \cdot (R+1) + R$ as the~$V_h$ have diameter at most~$R$ and~$\ch$ is honest. Hence, $d_G(u,v) \leq (R+1)\cdot d_H(\varphi(u),\varphi(v)) + R$.
\end{proof}

\section{Structure of the proof of Theorem~\ref{thm:FatK2t}} \label{sec:overview}

For the proof of Theorem~\ref{thm:FatK2t} we construct a graph-partition of a graph~$G$ with no $K$-fat $K_{2,t}$ minor, and then employ Lemma~\ref{lem part QI} to obtain the desired quasi-isometry. More precisely, we will prove the following stronger version of Theorem~\ref{thm:FatK2t}: 

\begin{theorem} \label{thm:FatK2t:GraphPartition}
    For every $t \in \N$ there exists a function $R: \N \rightarrow \N$ such that every graph $G$ with no $K$-fat $K_{2,t}$ minor has an honest, $R(K)$-bounded graph-partition over a graph~$H$ \st\ every $2$-connected multi-graph which is a minor of~$H$ is a $K$-fat minor of~$G$.
\end{theorem}

Let us first show that Theorem~\ref{thm:FatK2t:GraphPartition} implies Theorem~\ref{thm:FatK2t}:

\begin{proof}[Proof of \Tr{thm:FatK2t} given \Tr{thm:FatK2t:GraphPartition}]
    Fix $t, K \in \N$, and let~$G$ be a graph with no $K$-fat $K_{2,t}$ minor. Let $(H, (V_h)_{h\in V(H)})$ be an $R$-bounded graph-partition of~$G$ as provided by Theorem~\ref{thm:FatK2t:GraphPartition}. Then $G$ is $(R + 1, R/(R+1))$-quasi-isometric to $H$ by Lemma~\ref{lem part QI}, and $H$ has no $K_{2,t}$ minor.
\end{proof}

In this proof of \Tr{thm:FatK2t} we showed that~$G$ is quasi-isometric to the graph~$H$ from \Tr{thm:FatK2t:GraphPartition}. Since~$H$ has the property that all its $2$-connected minors are $K$-fat minors of~$G$, we have the following corollary: 

\begin{corollary} \label{cor:FatK2t:ForbiddingMoreGraphs}
    Fix $t \in \N$, and let $\cj$ be a class of finite, $2$-connected graphs containing $K_{2,t}$. 
    Then there exists a function $f: \N \rightarrow \N^2$ such that 
    every graph with no $K$-fat minor  in $\cj$ is $f(K)$-quasi-isometric to a graph with no minor in~$\cj$. 
    \qed
\end{corollary}

Our proof of Theorem~\ref{thm:FatK2t:GraphPartition} will be divided into two steps. The first step is to structure our graph $G$ as an $H$-partition as in Lemma~\ref{lem part QI}, but with additional properties (Lemma~\ref{lem:ObtainingaGPwithNiceProperties} below). The second step is to show that these properties imply that any 2-connected subgraph of $H$ is a $K$-fat minor of $G$ (Lemma~\ref{lem:GPwithNiceProperties}).
To describe these additional properties (\ref{itm:FatK2t:NEW:Layers}--\ref{itm:FatK2t:NEW:DistanceBtwBags} below), we need the following definitions.

A \defi{rooted} graph is a pair $(H,s)$ where $H$ is a graph and $s$ is one of its  vertices, called its \defi{root}. We will sometimes omit $s$ from the notation if it is clear from the context. A rooted graph $(H,s)$ has a natural layering: we denote by $L^i = L^i_{H,s} := \{h \in V(H) : d_H(s,h) = i\}$ the \defi{$i$-th layer} of $H$. Given a vertex $h \in V(H)$ we denote by $i_h = i_{h,s}$ the unique integer satisfying $h \in L^{i_h}$.

Let $\ch= (H, (V_h)_{h \in V(H)})$ be a graph-partition of a graph~$G$ over a  graph~$H$.
If $H$ is rooted, then for every $n \in \N$ we let $G^n=G^n_\ch$  denote the  subgraph of~$G$ induced by those vertices that are contained in partition classes~$V_h$ of nodes $h$ in the layers of $H$ up to $L^n$, i.e.\ $G^n := G[\bigcup_{i \leq n} \bigcup_{h \in L^i} V_h]$.

All graphs $H$ used in graph-partitions $\ch = (H, (V_h)_{h \in V(H)})$ in the remainder of this paper will be rooted, and we will ensure that 
\begin{enumerate}[label=\rm{(\roman*)}]
    \item \label{itm:FatK2t:NEW:Layers} for all $i \in \N$ the layer $L^i$ is an independent set,
\end{enumerate}
i.e.\ there are no edges $xy\in E(H)$ with $x,y\in L^i$.
In particular, $H$ is bipartite, and for every edge $gh \in E(H)$ there exists $i \in \N$ such that $g \in L^i$ and $h \in L^{i+1}$.

Given $\ch$ as above, and a node $h$ of $H$ which is not the root, we let \defi{$\partial^\downarrow_h$} be the set of vertices of $V_h$ that send an edge to some vertex of $G^{i_h-1}$.

The \defi{height $R_h$} of a node $h$ of $H$ is the maximum distance 
$\max_{v\in V_h} d(\partial^\downarrow_h,v)$ of one of its vertices from its `bottom' $\partial^\downarrow_h$. We say that $V_h$ is \defi{level}, if 
\begin{enumerate}[label=\rm{(\roman*)}]
\setcounter{enumi}{1}
  \item \label{itm:FatK2t:NEW:Bags} $V_h = B_{G-G^{ i_h-1}}(\partial^\downarrow_h, R_h)$, 
  \end{enumerate}
with the exception that for the root $s$ of $H$, we say that $V_s$ is \defi{level} if there is a vertex $o \in V_s$ such that $V_s = B_G(o, R_s)$. In that case, we assume that some such $o$ is fixed, and let \defi{$\partial^\downarrow_h$} be the singleton set containing $o$. In particular, $V_s$ then satisfies \ref{itm:FatK2t:NEW:Bags}. 

Recall that we are trying to produce a graph-partition \ch\ of our graph \g as in Theorem~\ref{thm:FatK2t:GraphPartition}, so that every 2-connected minor $J$ of $H$ is a $K$-fat minor of~\G. The naive way to try to turn $J<H$ into a $K$-fat minor of $G$ is to replace each vertex $h\in V(H)$ in the model of $J$ by $V_h$. 
But this is too naive for two reasons: firstly, the $V_h$ are not necessarily connected, and secondly, they are not necessarily $K$-far apart when we want them to be. To address these issues, instead of using a $V_h$ in our branch sets, we will instead use a connected region of $G$ around $\partial^\downarrow_h$. This region (depicted in (dark) blue in Figure~\ref{fig:FatK2t:Lemma1}) will consist of a subgraph of $V_h$ of height less than $R_h -K$, as well as an \defi{undergrowth}, i.e.\ a subgraph of the layer below $i_h$ (hence outside $V_h$) used to ensure connectedness. We use the following notation to describe these subgraphs precisely. For $h\in V(H)$ and $R\in \N$, let 
$$\partial^\uparrow_h(R):= B_{G-G^{i_h-1}}(\partial^\downarrow_h, R).$$
In particular, \ref{itm:FatK2t:NEW:Bags} can be reformulated as $V_h=\partial^\uparrow_h(R_h)$, but we will use this notation with $R< R_h$ to capture a shorter subgraph of $V_h$. To define the aforementioned undergrowth, we similarly introduce
$$\partial^\downarrow_h(r):= B_G(\partial^\downarrow_h, r) \setminus \partial_h^\uparrow(r)$$
for $h \in V(H)$ and $r \in \N$. We remark that we think of $\partial_h^\downarrow(r)$ as lying `below' $\partial^\downarrow_h$ and being mostly contained in $G^{i_h-1}$. In fact, whenever we use $\partial^\downarrow_h(r)$, we will make sure that for most other nodes $g \in V(H)$ in the same layer as $h$, their $\partial_g^\downarrow$ is more than $r$ far apart from $\partial_h^\downarrow$, so that $\partial^\downarrow_h(r)$ cannot enter $G^{i_h}$ through $\partial_g^\downarrow$ (and hence will be disjoint from $V_g$). (The only exception will be nodes $g \in L^{i_h}$ that can be separated from $h$ by removing a single node of $H$ (see \ref{itm:FatK2t:NEW:DistanceBtwBags} below).) 

\medskip
The second step of our proof of Theorem~\ref{thm:FatK2t:GraphPartition} mentioned above is made precise by the following lemma (see Figure~\ref{fig:FatK2t:Lemma1} for a sketch of the properties \ref{itm:FatK2t:NEW:Bags} to \ref{itm:FatK2t:NEW:DistanceBtwBags}):

\begin{lemma} \label{lem:GPwithNiceProperties}
    Let $K, \ell \in \N$, let $H$ be a rooted graph, and let~$G$ be a graph with an honest graph-partition $(H, (V_h)_{h \in V(H)})$ satisfying \ref{itm:FatK2t:NEW:Layers} and \ref{itm:FatK2t:NEW:Bags} \fe\ $h$.
    Suppose every $h \in V(H)$ has height $R_h\geq \ell+K$, and there is $r_h\in \N$ with $0 < r_h \leq \ell$ such that
    \begin{enumerate}[label=\rm{(\roman*)}]
        \setcounter{enumi}{2}
        \item \label{itm:FatK2t:NEW:Connected} $\partial^\uparrow_h(R_h - \ell - K) \cup \partial^\downarrow_h(r_h)$ is connected, and
        \item \label{itm:FatK2t:NEW:DistanceBtwBags} for all non-adjacent $g \neq h \in V(H)$ either $d_G(V_g, V_h) \geq 2\cdot\max\{r_g,r_h\}+3K$,
        or there is a node in $H$ that separates $g,h$.
    \end{enumerate}
    Then every $2$-connected subgraph of~$H$ is a $K$-fat minor of~$G$.
\end{lemma}

\begin{figure}[ht]
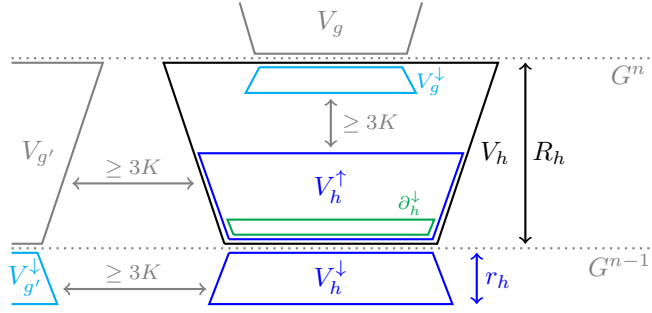

    \centering
    \include{Lemma_FatXMinor}
    \vspace{-2em}
    \caption{Depicted is a partition class $V_h$ of the graph-partition in Lemma~\ref{lem:GPwithNiceProperties}. The (dark) blue vertex set $V_h^\uparrow \cup V_h^\downarrow$ is connected by \ref{itm:FatK2t:NEW:Connected}, and $d_G(V_h \cup  V_h^\downarrow, V_{g'} \cup V^\downarrow_{g'}) \geq 3K$ holds by \ref{itm:FatK2t:NEW:DistanceBtwBags}.}
    \label{fig:FatK2t:Lemma1}
\end{figure}

Given the setup of this lemma, let $V_h^\uparrow:= \partial^\uparrow_h(R_h - \ell - K)$ and $V_h^\downarrow:= \partial^\downarrow_h(r_h)$. 
Thus \ref{itm:FatK2t:NEW:Connected} says that $V_h^\uparrow \cup V_h^\downarrow$ is connected. 

Let us briefly sketch how Lemma~\ref{lem:GPwithNiceProperties} is proved.  
Given a 2-connected $J \subseteq H$, we build a model of $J$ in \g by replacing each  vertex $h\in V(J)$  by $V_h^\uparrow \cup V_h^\downarrow$, which is connected by \ref{itm:FatK2t:NEW:Connected} as just mentioned. For each edge $e=hg\in E(J)$ where $g$ is in the layer above that of $h$, we model $e$ by a branch path within $V_h$ incident with the undergrowth $V^\downarrow_g$ of $g$ inside $V_h$. We have tuned our parameters (by demanding $r_h\leq \ell)$ so that each $V_h$ has a \defi{buffer zone} above $V_h^\uparrow$ and below all undergrowths protruding from the layer above, where it is safe to choose the branch paths (which are geodesics of length $K$).  
We then use \ref{itm:FatK2t:NEW:DistanceBtwBags} to show that the branch sets in $G$ are pairwise far apart.

\medskip
The final step in the proof of Theorem~\ref{thm:FatK2t:GraphPartition} will then be to show that if a graph does not contain $K_{2,t}$ as a fat minor, then it has a graph-partition satisfying \ref{itm:FatK2t:NEW:Layers} to \ref{itm:FatK2t:NEW:DistanceBtwBags} whose partition classes all have small radius. In fact, it will be more convenient to exclude $\Theta_t$ as a fat minor, where $\Theta_t$ denotes the multi-graph on two vertices with $t$ parallel edges. Note that $K_{2,t}$ can be obtained from $\Theta_t$ by subdividing each of its edges precisely once.

\begin{lemma} \label{lem:ObtainingaGPwithNiceProperties}
    There exists a function $R : \N^2_{\geq 1} \rightarrow \N$ satisfying the following. Let $t, K \in \N_{\geq 1}$ with $t \geq 3$, and let $G$ be a graph with no $K$-fat $\Theta_t$ minor. Then~$G$ admits an $R(t, K)$-bounded, honest graph-partition satisfying \ref{itm:FatK2t:NEW:Layers} to \ref{itm:FatK2t:NEW:DistanceBtwBags} for some $\ell \in \N$. 
\end{lemma}

Together, Lemmas~\ref{lem:GPwithNiceProperties} and~\ref{lem:ObtainingaGPwithNiceProperties} imply Theorem~\ref{thm:FatK2t:GraphPartition}:

\begin{proof}[Proof of Theorem~\ref{thm:FatK2t:GraphPartition}]
    If $K = 0$, then $G$ itself is $\Theta_t$-minor-free, and the graph-partition $(G, (V_g)_{g \in V(G)})$ with $V_g = \{g\}$ is as desired. So we may assume $K \geq 1$.
    For $t=0$, every graph excluding $K_{2,0}$ as a fat minor has bounded radius, and hence the assertion follows trivially.
    For $t = 1$, it is easy to see that every graph excluding $K_{2,1}$ as fat minor consists only of components that each have bounded diameter, and hence the assertion follows trivially. For $t=2$, the result follows from (the proof of) the $K_3$ case of Conjecture~\ref{conj fat min} (see \cite[Theorem~3.1]{GeoPapMin}) and Lemma~\ref{lem:SubdivingEdgesOfFatMinor}, where we note that in this case, $G$ admits a tree-partition over a tree $T$, which has no $2$-connected minors. Hence, we may assume $t \geq 3$.

    Since $K_{2,t}$ is not a $K$-fat minor of $G$, it follows by Lemma~\ref{lem:SubdivingEdgesOfFatMinor} that $\Theta_t$ is not a $3K$-fat minor of $G$.
    Let $(H, (V_h)_{h \in V(H)})$ be the graph-partition provided by Lemma~\ref{lem:ObtainingaGPwithNiceProperties} for $G, t, 3K$. Let $J$ be a $2$-connected (multi)-graph that is a minor of~$H$, and let $J'$ be an $\subseteq$-minimal subgraph of~$H$ which still contains~$J$ as a minor. It is straight forward to check that~$J'$ is $2$-connected.
    By Lemma~\ref{lem:GPwithNiceProperties}, $J'$ is a $K$-fat minor of~$G$, and so $J$ is a $K$-fat minor of $G$.
\end{proof}

\section{Proof of Lemma~\ref{lem:GPwithNiceProperties}} \label{subsec:FatK2tProof:Lemma1}

For every $h \in V(H)$, recall that 
\begin{align*}
    V_h^\uparrow &:= \partial^\uparrow_h(R_h - \ell - K), \text{ and} \\ 
    V_h^\downarrow &:= \partial^\downarrow_h(r_h) 
\end{align*}
(see Figure~\ref{fig:FatK2t:Lemma1}). 
In particular, $V_h^\uparrow \cup V_h^\downarrow$ is connected by \ref{itm:FatK2t:NEW:Connected}. Let us also remark that by \ref{itm:FatK2t:NEW:Bags} and because $r_h \leq R_h$, we have $V_h^\downarrow = B_{G-G^{i_h-2}}(\partial_h^\downarrow, r_h) \setminus V_h$. 
Let also $L^i := L^i_{H,s}$, for $i \in \N$, denote the $i$-th layer of $H$ with respect to its root $s$.

Let $J$ be a $2$-connected subgraph of~$G$. Our aim is to find a $K$-fat model of $J$ in $G$, and we start with the branch paths.
Let $f = gh \in E(J) \subseteq E(H)$. By \ref{itm:FatK2t:NEW:Layers}, we may assume that $h \in L^{i-1}$ and $g \in L^{i}$ for some $i \in \N$. 
Since $\ch$ is honest, there exists an edge $uv \in E(G)$ with $u \in V_{h}$ and $v \in V_{g}$, and hence $V^\downarrow_g \cap V_h \neq \emptyset$ since $r_g > 0$.
By \ref{itm:FatK2t:NEW:Bags}, observe that there exists a path $Q^f = q_0^f\dots q^f_{R_h+1}$, with $q_0^f=v$ and $q_1^f=u$, such that $q_1^f,\dots,q^f_{R_h+1}\in V_h$ and $q^f_{R_h+1}\in \partial_h^{\downarrow}$. Thus and since $r_g>0$, $\tilde{Q}^f = q^f_{r_g},\dots,q^f_{\ell+K+1}$ is a $V_g^{\downarrow}-V_h^{\uparrow}$ path contained in $V_h$ (of length $\ell+K+1-r_g\geq K+1$). 

We declare the initial segment $E_f := q^{f}_{r_g} \dots q^{f}_{r_g+K}$ of length $K$ of $\tilde{Q}^f$ to be the branch path corresponding to $f$. The remaining subpath $T_f:= q^{f}_{r_g+K} \dots q^f_{\ell+K+1}$ of $\tilde{Q}^f$ will be called the \defi{tentacle} of $f$, and we will make it part of the branch set below, to ensure that each branch path attaches to the branch sets of its end-vertices (see Figure~\ref{fig:FatK2t:FatXMinor}). 
    
To complete our construction of a model of $J$ in $G$, we now define the branch sets $U_x$ as follows:
for each $x \in V(J)$, let $F_x$ be the set of edges of $J$ that are incident with $x$ and whose other endvertex lies in $L^{i_x+1}$, and let (see Figure~\ref{fig:FatK2t:FatXMinor}) 
\begin{equation*}
    U_x := V^\uparrow_x \cup V^\downarrow_x \cup \bigcup_{e \in F_x} T_e \subseteq V_x \cup V_x^\downarrow.
\end{equation*}

We claim that these $U_x$ and $E_e$ form the branch sets and branch paths of a $K$-fat model of $J$ in $G$. 
\begin{figure}[ht]
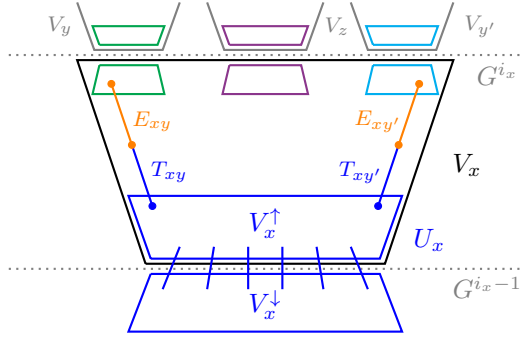

    \centering
    \include{Fat_K2t_Fat_XMinor}
    \vspace{-2em}
    \caption{Depicted is an illustration of~$U_x$, where $y,y',z \in V(J)$ and $xy, xy' \in E(J)$ and $xz \notin E(J)$.} 
    \label{fig:FatK2t:FatXMinor}
\end{figure}

By \ref{itm:FatK2t:NEW:Connected} and because $q^e_{\ell_e} \in V^\uparrow_x$ for all $e \in F_x$, the sets $U_x$ are connected.  
Since, by definition, every branch path $E_e$ of an edge $e = xy \in E(J)$, with $i_x < i_y$, starts in $q_0^{e} \in V^\downarrow_{y} \subseteq U_y$ and ends in $q^{e}_{K} \in V(T_{xy}) \subseteq U_x$, it follows that $((U_x)_{x \in V(J)}, (E_e)_{e \in E(J)})$ is a model of $J$ once we have shown that all pairs of non-equal and non-incident branch sets and/or paths are disjoint.
We will prove that they are even $K$-far apart in $G$, showing that our model of $J$ is $K$-fat.

For this, let us first note that since $J$ is $2$-connected, it follows by \ref{itm:FatK2t:NEW:DistanceBtwBags} that 
\[
d_G(V_x, V_y) \geq 2\cdot\max\{r_x,r_y\} + 3K 
\tag{$\ast$} \label{eq:FatK2t:Distance}
\]
for all $x, y \in V(J) \subseteq V(H)$ with $xy \notin E(H)$. 
In particular, 
\[
d_G(V_x \cup V^\downarrow_x, V_y \cup V^\downarrow_y) \geq 3K \tag{$\ast\ast$} \label{eq:FatK2t:Distance2}
\]
for all $x, y \in V(J) \subseteq V(H)$ with $xy \notin E(H)$.
    
Let $e = xy, e' = x'y' \in E(J)$ be distinct edges of $J$. 
Since $H$ is bipartite by \ref{itm:FatK2t:NEW:Layers} and hence triangle-free, and because $e \neq e'$, it follows that there are $a \in \{x, y\}$ and $b \in \{x', y'\}$ such that $a \neq b$ and $ab \notin E(H)$. Thus, by \eqref{eq:FatK2t:Distance2} and because $E_e$ and $E_{e'}$ meet $V_a \cup V^\downarrow_a$ and $V_b \cup V^\downarrow_b$, respectively, we have that
\[
d_G(E_e, E_{e'}) \geq d_G(V_{a} \cup V^\downarrow_a, V_{b} \cup V^\downarrow_b) - ||E_e|| - ||E_{e'}|| \geq 3K - K - K = K.
\]

Now let $z \in V(J)$ and $e = xy \in E(J)$ such that $z \notin \{x,y\}$. Once again, because $H$ is triangle-free, there exists $a \in \{x, y\}$ such that $za \notin E(H)$. Hence, as above, 
\[
d_G(U_z, E_e) \geq d_G(V_z \cup V^\downarrow_z, E_e) \geq d_G(V_z \cup V^\downarrow_z, V_{a} \cup V^\downarrow_{a}) - ||E_e|| \geq 3K - K \geq K.
\]
    
Finally, let $x \neq y \in V(J)$. 
If $xy \notin E(H)$, then, by \eqref{eq:FatK2t:Distance2},
\[
d_G(U_x, U_y) \geq d_G(V_x \cup V^\downarrow_x, V_y \cup V^\downarrow_y) \geq 3K \geq K,
\]
where we used that $U_z \subseteq V_{z} \cup V^\downarrow_{z}$ for all $z \in V(J)$.
    
So we may assume that $xy \in E(H)$. Then by \ref{itm:FatK2t:NEW:Layers} and without loss of generality, $x \in L^{i-1}$ and $y \in L^i$ for some $i \in \N$. 
If a shortest path in $G$ between $V^\uparrow_x \cup V^\downarrow_x$ and $V_y \cup V^\downarrow_y$ meets $G^{i-2}$, then by \ref{itm:FatK2t:NEW:Layers} it contains a subpath in $G-G^{i-2}$ from $\partial_h^\downarrow$ (for some $h \in L^{i-1}$) to $N_G(V_h) \cap V(G^{i}-G^{i-1})$. 
In that case, it follows by \ref{itm:FatK2t:NEW:Bags} that 
\[
d_{G}(V^\uparrow_x \cup V^\downarrow_x, V_y \cup V^\downarrow_y) \geq \min\{R_h : h \in L^{i-1}\} +1 \geq \ell+K+1 \geq K. 
\]
Otherwise, we have $d_{G}(V^\uparrow_x \cup V^\downarrow_x, V_y \cup V^\downarrow_y) = d_{G-G^{i-2}}(V^\uparrow_x \cup V^\downarrow_x, V_y \cup V^\downarrow_y)$, and hence, by \ref{itm:FatK2t:NEW:Bags},
\[
d_{G}(V^\uparrow_x \cup V^\downarrow_x, V_y \cup V^\downarrow_y) \geq d_{G-G^{i-2}}(V^\uparrow_x, \partial^\downarrow_y) - r_y \geq (\ell+K+1)-r_y. 
\]
Combining both cases we find
\[
d_{G}(V^\uparrow_x \cup V^\downarrow_x, V_y \cup V^\downarrow_y) \geq (\ell+K+1) - r_y \geq (\ell+K+1)-\ell \geq K.
\tag{$\ast\ast\ast$} \label{eq:FatK2t:Distance3}
\]

It remains to show that $d_G(T_e, V_y \cup V^\downarrow_y) \geq K$ for all edges $e \in F_x$. (Recall that all tentacles of $y$ are contained in $V_y$.) For this, let $e = xz \in E(J)$ with $e \in F_x$ be given. So $z \in L^i$. 
We split $T^e$ into an `upper part' $T^e_1 := V(T^e) \cap B_G(\partial^\downarrow_{z}, r_y + K)$ and a `lower part' $T^e_0 := V(T^e) \setminus B_G(\partial^\downarrow_{z}, r_y+K)$. 
Note that $T^e_1$ is empty if $r_z \geq r_y$, which is in particular the case when $y=z$.
We show separately that both $T^e_1,T^e_0$ have distance at least $K$ from $V_y \cup V^\downarrow_y$. 
Indeed, if $T^e_1$ is non-empty (and hence $z\neq y$), we have 
\begin{align*}
    d_G(T^e_1, V_y \cup V^\downarrow_y) &\geq d_G(B_G(\partial^\downarrow_{z}, r_y + K), V_y \cup V^\downarrow_y)
    \geq d_G(V_{z}, V_y) - (r_y + K) - r_y\\ 
    \intertext{since $\partial_{z}^\downarrow \subseteq V_{z}$ and $V^\downarrow_y \subseteq B_G(V_y, r_y)$. Hence, by \eqref{eq:FatK2t:Distance},}
    d_G(T^e_1, V_y \cup V^\downarrow_y)&\geq (2r_y + 3K) - r_y - K - r_y \geq K.
\end{align*}
Moreover, since $Q^e$ is a $V^\downarrow_{z}$--$V^\uparrow_{x}$ path of length $d_{G[V_x]}(V^\downarrow_{z}, V^\uparrow_x) = \ell+K+1-r_z$ in $G[V_x]$, we have $T^e_0 \subseteq B_G(V^\uparrow_{x}, \ell-r_y+1)$. 
It follows that
\begin{align*}
    d_G(T_0^e, V_y \cup V^\downarrow_y) &\geq d_G(V^\uparrow_x, V_y \cup V^\downarrow_y) - (\ell-r_y+1).\\
    &\geq (\ell+K+1-r_y) - (\ell-r_y+1) = K,
\end{align*}
where we used the first inequality of \eqref{eq:FatK2t:Distance3}. 
This concludes the proof that $d_G(U_x, U_y) \geq K$ for all $x \neq y \in V(J)$, and hence completes the proof of Lemma~\ref{lem:GPwithNiceProperties}. \qed

\begin{corollary} \label{cor:FirstLemma:PolyTimeAlgo}
    There is a polynomial-time algorithm that, given some $K \in \N$, a finite graph $G$, an $H$-partition of $G$ as in Lemma~\ref{lem:GPwithNiceProperties}, and a $2$-connected subgraph $J$ of $H$, returns a $K$-fat model of $J$ in $G$. 
\end{corollary}

\begin{proof}
    The above proof is constructive, and provides an efficient procedure to turn a subgraph $J$ of $H$ into a $K$-fat model of $J$ in $H$.
\end{proof}

\section{Component structure and \texorpdfstring{$K$}{K}-fat \texorpdfstring{$\Theta_t$}{Theta} minors} \label{subsec:FatK2tProof:FindingFatK2t}

The rest of the paper is devoted to the proof of Lemma~\ref{lem:ObtainingaGPwithNiceProperties}, for which we will construct a graph-partition of our graph~$G$ recursively. At the beginning of the $n$-th step of the recursion, we will already have constructed a graph-partition~$\ch^{n-1}$ of some induced subgraph~$G^{n-1}$ of~$G$. To proceed with the construction, we need that the components $C$ of $G-G^{n-1}$ satisfy two conditions. First, their boundaries $\partial_G C$ should not be too large, so that we can partition them into few sets of bounded radius. For this, we establish Lemma~\ref{lem:NearComps:NumberAndRadius} below, which finds a fat $\Theta_t$ minor otherwise.
Furthermore, we need that not too many components attach to the same bags of $\ch^{n-1}$. For this, we establish Lemma~\ref{lem:FatK2t:Components} below, which again finds a fat $\Theta_t$ minor otherwise. 

We start with a simpler lemma needed for both aforementioned lemmas.

\begin{lemma} \label{lem:FindingKFatK2tMinors}
    Let $G$ be a graph, and $K \in \N$. Let $X, Y \subseteq V(G)$ be connected and $d_G(X, Y) \geq K$.
    For every $t \in \N_{\geq 1}$, if $B_G(X, K) \cap V(Y)$ contains $t$ vertices which are pairwise at least~$3K$ apart, then~$\Theta_t$ is a $K$-fat minor of~$G$.

    Moreover, if \g is finite, then  there is a polynomial-time algorithm (for fixed~$t$) that given the above data either confirms that no such t-tuple of vertices exists, or returns a $K$-fat $\Theta_t$   minor of~$G$.
\end{lemma}

\begin{proof}
    Assume that $B_G(X, K) \cap Y$ contains vertices $u_1, \dots, u_t$ which are pairwise at least $3K$ apart in $G$. 
    For every $i \in [t]$, let $P_i$ be a $u_i$--$X$ path of length~$K$. 
    Then $V_1 := Y$ and $V_2 := X$ form the branch sets and the~$P_i$ form the branch paths of a $K$-fat model of~$\Theta_t$ in $G$. Indeed, we have $d_G(V_1, V_2) = d_G(X,Y) \geq K$ by assumption, and $d_G(P_i, P_j) \geq d_G(u_i, u_j) - ||P_i|| - ||P_j|| \geq 3K - K - K = K$. 
    \smallskip
    
    For the second claim, it is straightforward to efficiently check if $B_G(X, K) \cap Y$ contains such a $t$-tuple, as there are at most $n^t$ tuples to consider. If such a $t$-tuple is found, then the above proof provides an efficient procedure for finding a $K$-fat $\Theta_t$   minor.
\end{proof}

\begin{lemma} \label{lem:NearComps:NumberAndRadius}
    Let~$G$ be a graph, and let $X \subseteq V(G)$ be connected. Let further $K \in \N$, and let~$C$ be a component of $G - B_G(X, K-1)$. If~$\Theta_t$ is not a $K$-fat minor of~$G$ for some $t \geq 2$, then $\partial_G C$ has at most $t-1$ $3K$-near-components and each of them has diameter less than $6K(t-1)$.

    Moreover, if \g is finite, then there is a polynomial-time algorithm that  either confirms that $C$ has the aforementioned properties, or returns a $K$-fat $\Theta_t$ minor of~$G$.
\end{lemma}

\begin{proof}
    If $\partial_G C$ has at least $t$ $3K$-near components, then taking one vertex from each $3K$-near component yields $t$ vertices in $\partial_G C$ which are pairwise at least $3K$ apart. Applying \Lr{lem:FindingKFatK2tMinors} (with $X:=X$ and $Y:=V(C)$) yields that~$\Theta_t$ is a $K$-fat minor of~$G$.

    Now suppose that some $3K$-near component~$C'$ of $\partial_G C$ has diameter at least $6K(t-1)$, and pick vertices $u, v \in V(C')$ with $d_G(u,v) \geq 6K(t-1)$. Since~$C'$ is a $3K$-near component, there exists a $3K$-near path $P = x_0 \dots x_n$ in~$C'$ from $u = x_0$ to $v = x_n$. Let $W$ be an $u$--$v$ walk in~$G$ obtained from~$P$ by adding for every $i \in \{0, \dots, n-1\}$ an $x_i$--$x_{i+1}$ path of length at most~$3K$ to~$P$. 
    Since $d_G(u,v) \geq 6K(t-1)$, the walk~$W$ has vertices $u = y_1, y_2, \dots, y_{t-1}, y_t = v$ which are pairwise at least~$6K$ apart in~$G$. By the definition of~$W$, there exists for every~$y_j$ some~$x_{i_j}$ in~$P$, which hence lies in~$\partial_G C$, that has distance at most~$3K/2$ from~$y_i$. It follows that $d_G(x_{i_j},x_{i_\ell}) \geq d_G(y_j, y_\ell) - d_G(y_j, x_{i_j}) - d_G(y_\ell, x_{i_\ell}) \geq 6K - 3K = 3K$. Thus, applying \Lr{lem:FindingKFatK2tMinors} (with $X := X$ and $Y := V(C)$) to the~$x_{i_j}$ for $j \in [t]$ yields that~$\Theta_t$ is a $K$-fat minor of $G$.
    \smallskip

    For the second statement, it is again straightforward to compute and count the $3K$-near-components of $\partial_G C$, and to calculate their diameters, and so we can efficiently check whether $C$ satisfies the desired properties. If not, and the number of these $3K$-near-components is at least $t$, then invoking \Lr{lem:FindingKFatK2tMinors} as above will return a $K$-fat $\Theta_t$ minor. Finally, if one of these $3K$-near-components~$C'$ has diameter at least $6K(t-1)$, then the above proof yields an efficient procedure for finding a $t$-tuple of vertices in $C'$ pairwise at distance at least $3K$, and invoking \Lr{lem:FindingKFatK2tMinors} again returns a $K$-fat $\Theta_t$ minor. 
\end{proof}

Another consequence of Lemma~\ref{lem:FindingKFatK2tMinors} is
\begin{lemma} \label{lem:FatK2t:Components}
    Let $K, t, n \in \N$ with $t \geq 3$ and $n \leq t-1$, and let $G$ be a graph with no $K$-fat $\Theta_t$ minor. Let $X_1, X_2, \dots, X_n$ be connected subsets of $V(G)$ that are pairwise at least $3K$ apart and set $V' := \bigcup_{i \in [n]} B_G(X_i, K-1)$. 
    Let $\cc$ be the set of components of $G-V'$ that each have neighbours in at least two distinct $B_G(X_i, K-1)$. 
    Then there is no set of more than  $(t-1)^3(t-2)$ vertices of $\bigcup_{C \in \cc} \partial_G C$ pairwise at distance at least $3K$.

    Moreover, if \g is finite, then there is a polynomial-time algorithm that either confirms that $\cc$ has the aforementioned property, or returns a $K$-fat $\Theta_t$ minor of~$G$.
\end{lemma}

\begin{proof}
    Suppose for a contradiction that there is a set $U \subseteq \bigcup_{C \in \cc} \partial_G C$ of size at least $(t-1)^3(t-2)+1$ such that $d_G(u, u') \geq 3K$ for all $u, u' \in U$. 
    For every $u \in U$, let $C_u \in \cc$ be the component of $G-V'$ containing~$u$.

    By the pigeonhole principle and because $n \leq t-1$, there is $i \in [n]$ and a subset $U'\subseteq U$ of size at least $(t-1)^2(t-2)+1$ such that every $u \in U'$ has a neighbour in $B_G(X_i, K-1)$. Further, by the same argument and because every $C_u \in \cc$ has neighbours in at least two distinct~$B_G(X_j, K-1)$, it follows that there is $j \neq i \in [n]$ and a set $U'' \subseteq U'$ of size at least~$(t-1)^2+1$ such that for every $u \in U''$ the component~$C_u$ has a neighbour in~$B_G(X_j, K-1)$. 
    Moreover, by Lemma~\ref{lem:FindingKFatK2tMinors} (applied to $X := X_i$ and $Y := V(C_u)$ for every $u \in U''$) and because $\Theta_t$ is not a $K$-fat minor of~$G$, we deduce that there is a subset $W \subseteq U''$ of size at least $t$ such that $C_u \neq C_{u'}$ for all $u \neq u' \in W$. 
    
    We now use~$W$ to show that~$\Theta_t$ is a $K$-fat minor of~$G$, which contradicts our assumptions and thus concludes the proof.
    For every $u \in W$ pick a $u$--$X_i$ path $Q_u$ of length~$K$, which exists since $u \in N_G(B_G(X_i, K-1))$. 
    Then by the choice of $W$, the paths $Q_u$
    form the branch paths and $V_1 :=X_i$ and $V_2 := B_G(X_j, K-1) \cup \bigcup_{u \in W} V(C_u)$
    form the branch sets of a model of~$\Theta_t$ (see Figure~\ref{fig:FatK2t:Claim2}). 
    \begin{figure}[ht]
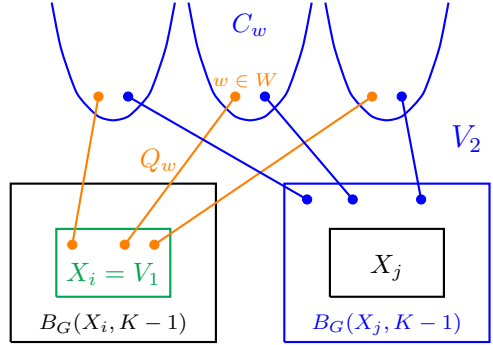

        \centering
        \include{FatK2t_Lemma1_Proof}
        \vspace{-2em}
        \caption{An illustration of the fat $\Theta_t$ minor in the proof of Lemma~\ref{lem:FatK2t:Components}. The green and blue sets are its branch sets, and the orange paths are its branch paths.}
        \label{fig:FatK2t:Claim2}
    \end{figure}
    We claim that this model is $K$-fat. Indeed, we have 
    \begin{align*}
    d_G(Q_u, Q_{u'}) \geq d_G(u, u') - ||Q_u|| - ||Q_{u'}|| \geq 3K - K - K = K,
    \end{align*}
    since $u, u' \in U$ and hence $d_G(u, u') \geq 3K$ by the assumption on $U$.
    Moreover,
    \[
    d_G(V_1, B_G(X_j, K)) \geq d_G(X_i, X_j) - K \geq 3K - K  > K
    \]
    by the assumption on the $X_k$.
    Finally, we have $d_G(V_1, C_u) \geq K$ for all $u \in W$ since $C_u$ is a component of $G-V'$, which concludes the proof.
    \smallskip

    For the second statement, it is straightforward to compute the set $\cc$ of components of $G-V'$ that each have neighbours in at least two distinct sets $B_G(X_i, K-1)$, and to check if $\bigcup_{C \in \cc} \partial_G C$ has such a $t'$-tuple of vertices where $t':=(t-1)^3(t-2)+1$, as there are at most $n^{t'}$ tuples to consider. If such a $t'$-tuple is found, then the above proof provides an efficient procedure for finding a $K$-fat $\Theta_t$ minor (where it might find the $K$-fat $\Theta_t$ minor by invoking Lemma~\ref{lem:FindingKFatK2tMinors}, which is the only point in the proof (except for the contradiction in the end) where we used the assumption that $\Theta_t$ is not a $K$-fat minor of $G$). 
\end{proof}

\section{A merging lemma} 
\label{sec:ML}

Recall that for the proof of Lemma~\ref{lem:ObtainingaGPwithNiceProperties} we will construct a graph-partition of a graph~$G$ recursively. After each step, we will have constructed a graph-partition $\ch^n$ of some subgraph of~$G$. In the next step, we will consider, for some suitable $K'\in\N$, the $K'$-near-components of the boundaries $\partial_G C$ of the remaining components~$C$ as candidates for the new partition classes which we aim to add to~$\ch^n$. However, some of the near-components might be too close to each other for~\ref{itm:FatK2t:NEW:DistanceBtwBags}, in which case we combine them into one new partition class. The following lemma formalises this merging procedure. When we apply the lemma, the set $\cq$ will be a candidate for the partition of the boundaries $\partial_G C$, the integer $r$ will be the minimum height which we want to achieve, and $L \geq r$ will be the height $R_h$ which we need to choose. Moreover, $d$ is the distance that we want to ensure between the new partition classes plus balls around them of radius $L$.

Given a set $U$ and partitions $\cp, \cq$ of $U$, we say that $\cp$ is a \defi{coarsening} of $\cq$ if every $B \in \cq$ is a subset of some $A \in \cp$. 

\begin{lemma} \label{lem:MergingLemma}
    Let $n \in \N$, let $G$ be a graph, and let~$\cq$ be a set of at most~$n$ disjoint subsets of $V(G)$. (We think of \cq\ as a partition of $\bigcup \cq$.) 
    Then for every $d,r \in \N$, there exist some $L \in \N$ with $r \leq L \leq r + \lfloor\frac{nd}{2}\rfloor$, and a coarsening~$\cp$ of $\cq$ such that 
    \begin{enumerate}[label=\rm{(\roman*)}]
        \item \label{itm:MergingLem:1} for every $A \in \cp$ and every $u,v \in A$ there is a sequence $(B_i)_{i \in [k]} \subseteq \cq$ with $B_i \subseteq A$ for all $i \in [k]$ such that $u \in B_1$, $v \in B_k$ and $d_G(B_{i-1},B_i)\leq 2L$ for all $i \in \{2, \dots,k\}$,
        \item \label{itm:MergingLem:3} $d_G(A, A') \geq 2L + d$ for all $A \neq  A' \in \cp$, and
        \item \label{itm:MergingLem:2} if $\diam_G(B) \leq D$ for all $B \in \cq$ and some $D \in \N$, then every $A \in \cp$ has diameter at most $nD + (n-1)(2r+nd)$.
    \end{enumerate}
    Moreover, if \g is finite, then \cp\ can be computed in polynomial time. 
\end{lemma}

\begin{proof}
    We first construct a coarsening $\cp$ satisfying \ref{itm:MergingLem:1} and \ref{itm:MergingLem:3}, and then verify that $\cp$ also satisfies \ref{itm:MergingLem:2}. We construct $\cp$ recursively as follows. Set $\cp_0 := \cq$ and $L_0 := r$, and assume that we have already defined $\cp_m$ for some $m < n$ such that~$\cp_m$ has $n-m$ elements and satisfies \ref{itm:MergingLem:1} with $L_m := r + \lfloor\frac{md}{2}\rfloor$ instead of~$L$.
    If~$\cp_m$ also satisfies \ref{itm:MergingLem:3} with~$L_m$ instead of~$L$, then $\cp:= \cp_m$ and $L := L_m$ are as desired. 
    In particular, if $m = n-1$, then $|\cp_m| = 1$, and hence~$\cp_m$ satisfies \ref{itm:MergingLem:3} trivially. 
    
    Otherwise, pick two sets $A, A' \in \cp_m$ with $d_G(A', A') < 2L_m + d$. 
    Then $\cp_{m+1} := (\cp_m \setminus \{A, A'\}) \cup \{A \cup A'\}$ has $n-m-1$ elements, and it still satisfies \ref{itm:MergingLem:1} with $L_{m+1} := r+\lfloor \frac{(m+1)d}{2}\rfloor \geq L_m + \lfloor \frac{d}{2}\rfloor$ instead of $L$. Indeed, let $a \in A$ and $a' \in A'$ such that $d_G(a, a') < 2L_m + d$. Then for every $u \in A$ and $v \in A'$ we can concatenate the sequences given by \ref{itm:MergingLem:1} for $u, a \in A$ and $a', v \in A'$, which yields a sequence for $u,v \in A \cup A'$ as in \ref{itm:MergingLem:1}.
    This completes the construction of~$\cp$ and the verification that~$\cp$ satisfies \ref{itm:MergingLem:1} and \ref{itm:MergingLem:3}.
    \smallskip

    To check \ref{itm:MergingLem:2}, let $A \in \cp$, and assume that $\diam_G(B) \leq D$ for some $D \in \N$ and all $B \in \cq$. Then 
    \begin{align*}
    \diam_G\left(A\right) \leq nD +(n-1)2L.
    \end{align*}
    by picking $u,v \in A$, and a sequence of $B_i's$ as in \ref{itm:MergingLem:1}, and noting that we have at most $n$ such $B_i's$. The right hand side is at most $nD+ (n-1)(2r+ nd)$ by our bound on $L$.
\smallskip

Since this recursive construction terminates after at most $n$ steps, each of which only compares distances between pairs of at most $n$ sets of vertices, it can be carried out by a polynomial-time algorithm. 
\end{proof}

\section{Proof of Lemma~\ref{lem:ObtainingaGPwithNiceProperties}}
\label{subsec:FatK2tProof:Lemma2}

We can now prove Lemma~\ref{lem:ObtainingaGPwithNiceProperties}. We will provide concrete values for $R(t,K)$ and $\ell$ that satisfy our requirements, but the reader can choose to ignore these values; what matters is that we can choose $R(t,K), \ell$ large in comparison to $t$ and~$K$, more concretely, large enough compared to values that come out of applications of Lemmas~\ref{lem:NearComps:NumberAndRadius}, ~\ref{lem:FatK2t:Components} and~\ref{lem:MergingLemma}. The values that we obtain are\footnote{We remark that we rounded the function $R_0(t, K)$ up to make it more readable. It is much larger than $N(t)$ and $L(t,K)$ but independent of $L'(t,K)$.} 
\begin{align*}
    N(t) &:= \bigg\lceil\frac{1}{2}(t-1)^3\cdot(t-2)\bigg\rceil,\\
    L(t,K) &:= \lceil3K/2\rceil + N(t) \cdot 3K,\\
    L'(t,K) &:= N(t) \cdot \big(4\cdot L(t,K) + 5K\big) 
    + 2\cdot L(t,K)+3K,\\
    R_0(t,K) &:=  15t^{12}K+18t^9K, 
    \text{ and}\\
    R(t,K) &:= R_0(t,K) + 2L'(t,K) \in O(t^{12}K). 
\end{align*}
We prove Lemma~\ref{lem:ObtainingaGPwithNiceProperties} with the function $R(t,K)$ and $\ell := L(t,K)$. 
\medskip

Let $t, K \in \N_{\geq 1}$ with $t \geq 3$, and let $G$ be a graph with no $K$-fat $\Theta_t$ minor. 
By considering each component of $G$ individually, we may assume that $G$ is connected.
\medskip

We first describe a method to inductively define a bounded and honest graph-partition that satisfies \ref{itm:FatK2t:NEW:Layers} and~\ref{itm:FatK2t:NEW:Bags}. We then fix specific constants so that also \ref{itm:FatK2t:NEW:Connected} and \ref{itm:FatK2t:NEW:DistanceBtwBags} hold.

We construct the desired graph-partition $\ch = (H, (V_h)_{h \in V(H)})$ of $G$ recursively `layer by layer', i.e.\ the nodes that we add to $H$ in the $n$-th step of the construction will form the $n$-th layer $L^n := L^n_{H,s}$ of $H$ with respect to the root~$s$ of $H$, which we specify in the first construction step. 

Pick $o \in V(G)$ arbitrarily.
We initialize $H^0 := (\{s\}, \emptyset)$ on a single vertex~$s$, its root, and set $V_{s} := B_G(o, L'(t,K))$. Then $H^0 = (H^0, (V_s))$ is an honest graph-partition of $G^0 = G[V_s]$. Moreover, $L^0 = \{s\}$. 

Having defined graph-partitions $\ch^i$ of $G^i$ \fe\ $i\leq n$, we proceed to construct $\ch^{n+1}$.
The main effort will go into finding a suitable partition $\cp$ of $N_G(G^n)$ into sets of diameter at most $R_0(t,K)$ (whose construction we postpone for later).
The new vertices of $H^{n+1}-H^n$ will be in bijection with the elements of $\cp$. For each $A \in \cp$, we introduce a vertex $h_A$, fix a `height' $R_A = R_{h_A} \leq L'(t, K)$.
We choose $\cp$ and the heights $R_A$ so that the $V_{h_A}$ are pairwise disjoint, and there is no edge of~$G$ between $V_{h_A}$ and $V_{h_B}$ for $A \neq B \in \cp$ (in fact, the $V_{h_A}$ will be pairwise far apart; see \ref{itm:FatK2t:DistanceBtwBags:Copy} below). 
We add an edge between nodes $h, h'\in V(H^{n+1})$ whenever there is an edge in $G$ between $V_h$ and $V_{h'}$. By the last property, $L^{n+1} = V(H^{n+1}-H^n)$ is independent. Moreover, $L^n = V(H^n-H^{n-1})$ separates $L^{n+1}$ from all $L^i$ with $i \leq n-1$ since the partition classes of nodes $h \in L^n$ contain the neighbourhood of $G^{n-1}$. Hence, $V(H^{n+1}-H^n)$ is indeed the $(n+1)$st layer $L^{n+1}$ of $H^{n+1}$.
By definition, $\ch^{n+1} = (H^{n+1}, (V_h)_{h \in V(H^{n+1})})$ is an honest graph-partition of $G^{n+1} = G[\bigcup_{h \in V(H^{n+1})} V_h]$.

If $N_G(G^n)$ is empty at some step $n$, which happens precisely when $G$ has finite diameter, then the process terminates. This is the only difference between the finite and infinite diameter case throughout our proof. 

We let $H := \bigcup_{n \in \N} H^n$. Then $\ch := (H, (V_h)_{h \in V(H)})$ is an honest graph-partition of $\bigcup_{n \in \N} G^n$, which is equal to $G$ since $G$ is connected and each $G^{n+1}$ contains the neighbourhood of $G^n$. By the comment above, $\ch$ satisfies \ref{itm:FatK2t:NEW:Layers}. 
Furthermore, $\ch$ satisfies \ref{itm:FatK2t:NEW:Bags} by the definition of $V_{h_A}$ and because $\partial^\downarrow_{h_A} = A$.
Moreover, as every $A \in \cp$ has diameter at most $R_0(t,K)$ and $R_A \leq L'(t,K)$, every partition class $V_{h_A}$ of $\ch$ has diameter at most $R_0(t,K) + 2L'(t,K) = R(t, K)$, and hence $\ch$ is $R(t,K)$-bounded.
\medskip

Thus, it only remains to specify $\cp$ and the heights $R_A$, which we will choose so that $R_A \geq 2\ell+3K$, and check that \ref{itm:FatK2t:NEW:Connected} and \ref{itm:FatK2t:NEW:DistanceBtwBags} hold. 

We repeat these properties here: 
for all $h \in V(H)$
\begin{enumerate}[label=\rm{(\arabic*)}]
    \item \label{itm:FatK2t:Connected:Copy} $\partial^\uparrow_h(R_h - \ell - K) \cup \partial^\downarrow_h(r_h)$ is connected.
\end{enumerate}
(We will specify the `depths' $r_h \leq \ell$ later on.) 
Recall that $\partial^\downarrow_h$ is the set of all vertices of $V_h$ that have a neighbour in $G^{i_h-1}$; in particular, $\partial^\downarrow_{h_A} = A$ for every node $h_A \in L^{n+1}$ by definition. We need the following modified version of \ref{itm:FatK2t:NEW:DistanceBtwBags}:
\begin{enumerate}[label=\rm{(\arabic*)}]
    \setcounter{enumi}{1}
    \item \label{itm:FatK2t:DistanceBtwBags:Copy} for all non-adjacent $g \neq h \in V(H)$ either $d_G(V_g, V_h) \geq 2\cdot\max\{r_g,r_h\}+3K$ or there is a node $x \in V(H)$ such that $V_x$ separates $V_g,V_h$ in $G$.
\end{enumerate}
(Note that \ref{itm:FatK2t:DistanceBtwBags:Copy} immediately implies \ref{itm:FatK2t:NEW:DistanceBtwBags} since $\ch$ is honest.) 

For our construction we need to inductively ensure that \ref{itm:FatK2t:Connected:Copy} and \ref{itm:FatK2t:DistanceBtwBags:Copy} hold for all $g,h \in V(H^n)$.
We remark that while for \ref{itm:FatK2t:Connected:Copy} it is enough to ensure that every $\ch^n$ is a graph-partition satisfying \ref{itm:FatK2t:Connected:Copy} \emph{with respect to $G^n$}, we need that $\ch^n$ satisfies \ref{itm:FatK2t:DistanceBtwBags:Copy} \emph{within~$G$}, i.e.\ if two partition classes $V_g, V_h$ of nodes $g\neq h \in V(H^n)$ are too close \emph{in $G$}, then the partition class $V_x$ of some node $x \in V(H^n)$ separates $V_g, V_h$ \emph{in~$G$}.

Moreover, we need to inductively ensure that the following property is true:
\begin{enumerate}[label=\rm{(\arabic*)}]
    \setcounter{enumi}{2}
    \item \label{itm:FatK2t:Star} Every component $C$ of $G-G^{n-1}$ meets at most $t-1$ partition classes $V_{h}$ of $\ch^n$.
\end{enumerate}
Letting $R_{s} := L'(t,K)$, $r_s:=0$, and $G^{-1}:= \emptyset$ clearly satisfies \ref{itm:FatK2t:Connected:Copy} to \ref{itm:FatK2t:Star} for $n=0$.

For every component $Z$ of $G-G^{n-1}$ let $\cd_Z$ be the set of all components of $G-G^n$ that are contained in $Z$ and that have neighbours in at least two distinct partition classes of $\ch^n$. Recall that $\cc(G-G^n)$ is the set of components of $G-G^n$.
Let $\mathfrak{R}$ be the partition of $\cc(G-G^n)$ comprising the $\cd_Z$ as above and a singleton $\{C\}$ for each component~$C$ of $G-G^n$ that is not in any $\cd_Z$ (i.e.\ that has neighbours in exactly one partition class of $\ch^n$) (see Figure~\ref{fig:PartitionR}). 

\begin{figure}[ht]
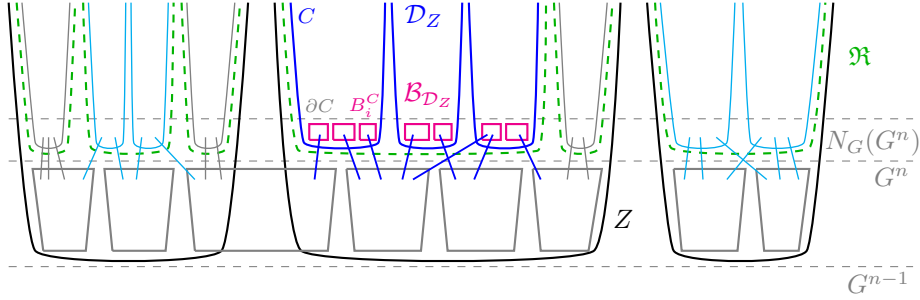

    \centering
    \include{Partition_R}
    \vspace{-3em}
    \caption{A visualisation of the partition $\mathfrak{R}$ of $\cc(G-G^n)$ (in green, with dashed lines). Every partition class in $\mathfrak{R}$ is either a singleton comprising a component that has only neighbours in exactly one partition class of $\ch^n$ (indicated in grey), or it is $\cd_Z$ for some component $Z$ of $G-G^{n-1}$ (indicated in light/dark blue).}
    \label{fig:PartitionR}
\end{figure}

Note that $\mathfrak{R}$ naturally induces a partition $\cgr$ of $N_G(G^n)$, by letting  $\cgr:= \big\{\partial_G (\bigcup \cd) \mid \cd\in \mathfrak{R}\big\}$. 
We will obtain $\cp$ by refining $\cgr$.

\medskip
For every $C \in \cc(G-G^n)$, let $B^C_1, \dots, B^C_{m_C}$ be the $3K$-near components of~$\partial_G C$. We group these $3K$-near components together over $\cgr$ by considering 
\[
\cb_\cd := \{B^C_i : C \in \cd,\, i \leq m_C\} \text{ for every } \cd \in \mathfrak{R}.
\]
Set $\cb := \bigcup_{\cd \in \mathfrak{R}} \cb_\cd$, and note that $\bigcup \cb = N_G(G^n)$. 
Our final partition $\cp$ of $N_G(G^n)$ will be a refinement of $\cgr$ and a coarsening of $\cb$.

We may think of $\cb$ as candidate for the partition $\cp$ of $N_G(G^n)$, and the $B_i^C$ as candidates for the new partition classes $V_h$ that we want to add to $\ch^{n}$. By taking $r_h := 0$ for all such new $V_h$ (and $R_h = 0$), they would already satisfy \ref{itm:FatK2t:DistanceBtwBags:Copy} at least in the case where $B_i^C$ and $B_j^{C'}$ are from the same component $C=C' \in \cd$. However, we need that they also satisfy \ref{itm:FatK2t:DistanceBtwBags:Copy} for $C \neq C'$. Moreover, since the $B_i^C$ are only $3K$-near components, they need not be connected, and hence might also not satisfy \ref{itm:FatK2t:Connected:Copy}. To make them connected, we might have to increase the heights $R_h$ and `depths' $r_h$ to $3K/2$. By doing so, the~$B_i^C$ might no longer satisfy \ref{itm:FatK2t:DistanceBtwBags:Copy} even within the same component $C$. To solve these two problems, we have to merge $B_i^C$'s that are to close. For this, we will employ Lemma~\ref{lem:MergingLemma}, which ensures that the merged sets are far apart and have bounded diameters (see Figure~\ref{fig:MergingLemma:1}). 
\begin{figure}[ht]
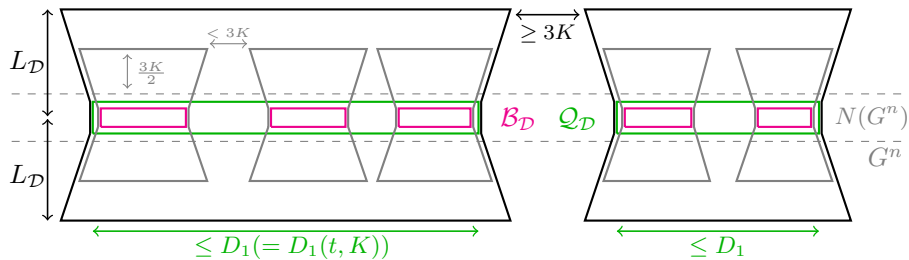

    \centering
    \include{MergingLemma1}
    \vspace{-3em}
    \caption{Indicated in pink is the partition $\cb_\cd$.
    The grey boxes around the $B_i^C$ are connected, but they need not to be pairwise $3K$ far apart. Applying Lemma~\ref{lem:MergingLemma} yields a coarsening $\cq_\cd$ of $\cb_\cd$ such that the black boxes (of height $L_\cd$) around the (green) partition classes of $\cq_\cd$ are connected and pairwise at least $3K$ far apart. Moreover, the partition classes in $\cq_\cd$ have bounded diameter.}
    \label{fig:MergingLemma:1}
\end{figure}
In order to apply Lemma~\ref{lem:MergingLemma}, we need to ensure that each $\cb_\cd$ contains only boundedly many elements all of bounded diameter. More precisely, we claim that for all $\cd \in \mathfrak{R}$
\labtequtag{cb}{there is no set of more than $(t-1)^3(t-2)$ elements of $\cb_\cd$ that are pairwise at least $3K$ far apart,
}{$\ast$}
(we will later merge elements of $\cb_\cd$ so that this bound will play the role of $n$ in our application of Lemma~\ref{lem:MergingLemma}), and
\labtequtag{Diametercb}{$\diam_G(B) \leq 6K(t-1)$ for every $B \in \cb_\cd$.
}{$\ast\ast$}
For this, let $C \in \cc(G-G^n)$. 
Then applying Lemma~\ref{lem:NearComps:NumberAndRadius} to $C$ (with some $X$ that we specify in the next sentence) yields that every $3K$-near component $B_i^C$ of $\partial_G C$ has diameter at most $6K(t-1)$ in $G$ and that $m_C \leq t-1$; in particular, \eqref{Diametercb} holds.
For this, let $X$ be the component of $Y:=G^n-B_G(G-G^n, K-1)$ which contains $o$ (which exists, since $V(G^0) = B_G(o, L'(t,K))$ and $L'(t,K)>K-1$). 
For the application of Lemma~\ref{lem:NearComps:NumberAndRadius}, we need to check that $C$ is a component of $B_G(X, K-1)$, which we do next. Since $C$ is connected and avoids~$X$ (as $X\subseteq G^n$ and $C\subseteq G-G^n$), it suffices to show that $N_G(C) \subseteq B_G(X, K-1)$. Pick $v \in N_G(C)$,
and note that $v\in V(G_n)$, so there is some $h\in L_i$ with $v\in V_h$. By \ref{itm:FatK2t:NEW:Bags} of $G^n$ (which holds inductively), we obtain $i=n$ and $\dist_{G_n}(v,\partial_h^{\downarrow})=R_h\geq K$. Let $P$ be a shortest $v$--$\partial_h^{\downarrow}$ path. By \ref{itm:FatK2t:NEW:Bags}, exactly the first $K-1$ vertices of $P$ are not in $Y$, and the $K$-th vertex $w$ (which exists) is contained in $Y$ and has distance exactly $K-1$ from $v$.
We need to show that $w \in X$.
For this, note first that the remainder of $P$ provides a $w$--$\partial_h^{\downarrow}$ path in $Y$.
Since all vertices in $\partial^\downarrow_h$ send an edge to $G^{i-1}$ by the definition of $\partial^\downarrow_h$, there is a path in $Y$ from $w$ to $G^{i-1}$ (note that $V(G^{i-1}) \subseteq Y$). Inductively applying this argument (except that now the entire path $P$ is contained in $Y$) thus yields that there is a $w$--$o$ path in $Y$, and thus $w \in X$ as claimed.

To complete the proof of \eqref{cb}, note that if $\cd = \{C\}$ for some $C \in \cc(G-G^n)$, then \eqref{cb} follows immediately from $m_C \leq t-1$. 
Otherwise, $\cd = \cd_Z$ for some component $Z \in \cc(G-G^{n-1})$. We then obtain \eqref{cb} by applying Lemma~\ref{lem:FatK2t:Components} with the sets $X_i$ being the sets $\partial^\uparrow_h(R_h-K-1) \cup \partial^\downarrow_h(r_h)$ for nodes $h \in L^n$ whose partition class $V_h$ has a neighbour in some $C \in \cd$ (which implies that $\cd=\cd_Z$ is a subset of the set $\cc$ from Lemma~\ref{lem:FatK2t:Components}). For this, note that there are at most $t-1$ such $V_h$ by \ref{itm:FatK2t:Star} and because the components in $\cd = \cd_Z$ are all contained in $Z$, and hence every such $V_h$ meets $Z$ (at least in a vertex of $N_G(C)$). 
Moreover, note that the $X_i$ are connected by \ref{itm:FatK2t:Connected:Copy} and are pairwise at least $3K$ apart by~\ref{itm:FatK2t:DistanceBtwBags:Copy} (as $X_i, X_j \subseteq V(Z)$ implies that no partition class separates them). 
\medskip
 
Having established the conditions \eqref{cb} and \eqref{Diametercb}, we are almost in a position to apply Lemma~\ref{lem:MergingLemma}, except that the size of the $\cb_\cd$'s is not yet bounded. For this, we modify $\cb_\cd$ for $\cd \in \mathfrak{R}$ as follows. Let $\cb'_\cd$ be some maximal subset of $\cb_\cd$ such that every two elements of $\cb'_\cd$ are at least $3K$ apart in $G$. We now obtain $\cb''_\cd$ by merging every $B \in \cb_\cd\setminus \cb'_\cd$ to a single (but arbitrary) $B' \in \cb'_\cd$ from which it has distance less than $3K$. Then $\cb''_\cd$ is a coarsening of $\cb_\cd$ and has size at most $(t-1)^3(t-2)$ by \eqref{cb}. Moreover, every $B \in \cb''_\cd$ has diameter at most $3 \cdot (6K(t-1)) + 2\cdot(3K-1) = 18tK - 12K-2$.

We can now apply Lemma~\ref{lem:MergingLemma} to each~$\cb''_\cd, \cd \in \mathfrak{R}$, with the parameters being $n :=  |\cb''_\cd| \leq |\cb_\cd| \leq 2N(t) \leq t^4$, $r := \lceil3K/2\rceil$, $d := 3K$ and $D := 18tK-12K-2 \leq 18tK-3K$. 
This merging yields a coarsening~$\cq_\cd$ of~$\cb''_\cd$ and some $L_\cd \leq \ell$ (see Figure~\ref{fig:MergingLemma:1}) such that every $A \in \cq_\cd$ has diameter at most $D_1 := nD + (n-1)(2r+nd) \leq 3t^8K+18t^5K$ (by \ref{itm:MergingLem:2}) and such that $B_G(A, L_\cd)$ is connected (by \ref{itm:MergingLem:1}, because $B_G(B, \lceil 3K/2\rceil)$ is connected for every $B \in \cb''_\cd$, and because $L_\cd \geq r = \lceil 3K/2\rceil$). Moreover, (by \ref{itm:MergingLem:3}) for all $A, A' \in \cq_\cd$
\labtequtag{distance1}{$d_{G}(A, A') \geq 2L_{\cd}+3K$.}
{$\Box$}
Set $\cq := \bigcup_{\cd \in \mathfrak{R}} \cq_\cd$, and note that $\bigcup \cq = \bigcup \cb = N_G(G^n)$.

The partition $\cq$ is our new candidate for $\cp$, and the $L_\cd$ are our candidates for the `heights' $R_A$. They would satisfy \ref{itm:FatK2t:Star} and a variant of \ref{itm:FatK2t:DistanceBtwBags:Copy} (see \eqref{distance1}), and they would satisfy a variant of \ref{itm:FatK2t:Connected:Copy} with `depths' $r_h := L_\cd$ whereby we need the whole height for connectedness, i.e.\ for all $A \in \cq_\cd$ we have that
\labtequtag{fullheightconn}{$B_G(A, L_\cd)$ is connected,}
{$1'$}
which we have proven above. 
Note that $B_G(A, L_\cd)$ would be equivalent to $\partial_h^\uparrow(R_h) \cup \partial_h^\downarrow(r_h)$
if we would set $R_h,r_h := L_\cd$ and $V_h := B_{G-G^n}(A, R_h)$.
To achieve \ref{itm:FatK2t:Connected:Copy}, we need to add a `buffer zone' of height $\ell+K$, that is, we need to increase the `height' $R_A$ for each $A \in \cq$ by $\ell+K$. This increase in height might however violate \ref{itm:FatK2t:DistanceBtwBags:Copy} even if \ref{itm:FatK2t:DistanceBtwBags:Copy} was satisfied earlier, and therefore we need to perform another round of merging, namely to merge any sets in some $\cq_\cd$ that violate \ref{itm:FatK2t:DistanceBtwBags:Copy}, i.e. which are two close together (see Figure~\ref{fig:MergingLemma:2}). This merging will ensure \ref{itm:FatK2t:DistanceBtwBags:Copy}, and \ref{itm:FatK2t:Connected:Copy} will follow from \eqref{fullheightconn}, as we will see below.

\begin{figure}[ht]
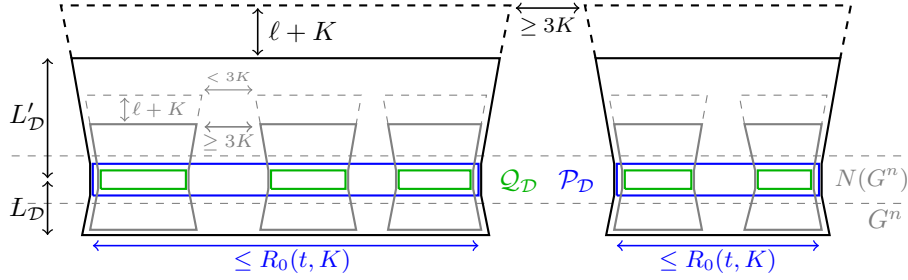

    \centering
    \include{MergingLemma2}
    \vspace{-3em}
    \caption{Indicated in green is the partition $\cq_\cd$. The grey boxes around its partition classes (of height $L_\cd$) are connected and pairwise at least $3K$ far apart, but to ensure \ref{itm:FatK2t:Connected:Copy}, we need to add a `buffer zone' of height $\ell+K$ (indicated with dashed lines). These taller boxes need no longer be pairwise $3K$ far apart. Applying Lemma~\ref{lem:MergingLemma} yields a coarsening $\cp_\cd$ of $\cq_\cd$ such that the black boxes (of `depth' $L_\cd$ and height $L'_\cd$) around the (blue) partition classes of $\cp_\cd$ are connected, and such that they are still $3K$ far apart even after adding a `buffer zone' of height $\ell+K$. Moreover, the partition classes in $\cq_\cd$ have diameter at most $R_0(t,K)$.}
    \label{fig:MergingLemma:2}
\end{figure}

To perform the aforementioned merging, we now apply Lemma~\ref{lem:MergingLemma} again, to each~$\cq_\cd$ with $\cd \in \mathfrak{R}$. More precisely, we apply Lemma~\ref{lem:MergingLemma} to $\cq_\cd$ in the subgraph $G-G^{n}$ with $n' := |\cq_\cd| \leq 2N(t)$, $r' := \ell+2K$, $d' := 4 \ell+5K$ and $D' := D_1$. 
This yields a coarsening~$\cp_\cd$ of~$\cq_\cd$ and some $L'_\cd \leq L'(t,K)-\ell-K$ with $L'_\cd \geq r'$ (see Figure~\ref{fig:MergingLemma:2}). 
This new $L'_\cd$ is the `height' that we need to ensure connectedness as in \eqref{fullheightconn} (or in \ref{itm:FatK2t:Connected:Copy}), i.e.\ for all $A \in \cp_\cd$ it follows by \eqref{fullheightconn} and Lemma~\ref{lem:MergingLemma}~\ref{itm:MergingLem:1} that
\labtequtag{fullheightconn2}{$B_{G-G^n}(A, L'_\cd) \cup B_G(A, L_\cd)$ is connected.}
{$1''$}
Moreover, by Lemma~\ref{lem:MergingLemma} \ref{itm:MergingLem:3}, for every $A \neq B \in \cp_\cd$,
\labtequtag{distance}{$d_{G-G^n}(A, A') \geq 2L'_{\cd}+4 \ell+5K$.}
{$\triangle$}

Setting $\cp := \bigcup_{\cd \in \mathfrak{R}} \cp_\cd$, we have defined our desired partition of~$N_G(G^n)$. Note that $\cp$ is a refinement of $\cgr$ and a coarsening of $\cb$. Moreover, every $\cp_\cd$ is a coarsening of $\cb''_\cd$, and hence of $\cb_\cd$. Since $\bigcup \cb_\cd = \bigcup_{C \in \cd} \partial_G C$, every $A \in \cp$ is contained in the union of the boundaries of components in some $\cd_A \in \mathfrak{R}$.

For every $A \in \cp$, we set $R_A := L'_{\cd_A}+\ell+K$ and $r_A := L_{\cd_A}$. Note that $2\ell + 3K = r'+\ell+K \leq R_A \leq L'(t,K)$ and $0 < r \leq r_A \leq \ell$. This completes the construction at step $n+1$. 

Let us note that, as promised in the description of our construction in the beginning, the partition classes $V_A$ for $A \in \cp$ are pairwise disjoint and not joined by edges. Indeed, if $A \neq A' \in \cp$ do not meet the same component of $G-G^n$, then this is immediate. Otherwise, $A,A'$ are both contained in the same $\cp_\cd$, and hence this follows from \eqref{distance}.
\medskip

It remains to check that every $A \in \cp$ has diameter at most $R_0(t,K)$ and that $\ca$ and the $R_A, r_A$ satisfy \ref{itm:FatK2t:Connected:Copy} to \ref{itm:FatK2t:Star}. 
For every $A \in \ca$ we have\footnote{We remark that this is the (only) inequality that $R_0(t,K)$ needs to satisfy. Since $n', D', r'$ and $d'$ depend only on $t,K$ and $\ell$ (which in turn depends only on $t$ and $K$ as $\ell := L(t,K)$), it suffices to choose $R_0(t,K)$ large in comparison to $t$ and $K$.} \footnote{We used here that $n' \leq 2N(t) \leq t^4$ and $r' := \ell+2K \leq 3t^4K$ and $d' := 4\ell+5K\leq 12t^4K$ and $D':=D_1 \leq 3t^8K+18t^5K$ (where we used for $r'$ and $d'$ that $N(t) \leq t^4-2$).} 
\[
    \diam_G(A) \leq n'D' + (n'-1)(2r'+n'd') \leq R_0(t,K)
\]
by Lemma~\ref{lem:MergingLemma}~\ref{itm:MergingLem:2}.
\smallskip

To prove \ref{itm:FatK2t:Connected:Copy}, let $A \in \cp$ and $h:=h_A$.
By the choice of $R_h, r_h$ we have $R_h = L'_{\cd_A}+\ell+K$ and $r_h = L_{\cd_A}$, and hence \ref{itm:FatK2t:Connected:Copy} follows from \eqref{fullheightconn2}.
\smallskip

To prove \ref{itm:FatK2t:Star}, let $C$ be a component of $G-G^n$. Since $\cp_\cd$ is a coarsening of $\cb_\cd$ and $\cp$ is the union over all $\cp_\cd$ with $\cd \in \mathfrak{R}$, there are at most $m_C$ elements of $\cp$ that meet $C$. By the definition of the new partition classes $V_{h_A}$ as $B_{G-G^n}(A, R_A)$, it follows that at most $m_C$ partition classes of $\ch^{n+1}$ meet $C$. Since $m_C \leq t-1$ as shown earlier, this concludes the proof of \ref{itm:FatK2t:Star}.
\smallskip

To prove \ref{itm:FatK2t:DistanceBtwBags:Copy}, let $g \neq h \in V(H^{n+1})$ be non-adjacent. By \ref{itm:FatK2t:DistanceBtwBags:Copy} of $H^n$, it suffices to consider the case where $g \in L^{n+1} = V(H^{n+1}-H^{n})$.
If $h \in V(H^{n-1})$, then 
\begin{equation} \label{eq:a}
\begin{aligned}
d_G(V_g, V_h) &\geq d_G(G-G^{n}, G^{n-1}) \geq \min\{R_{h'} : h' \in L^n\}\\ 
&\geq 2 \ell + 3K \geq 2\cdot\max\{r_g,r_h\} + 3K.
\end{aligned} \tag{a}
\end{equation}
where the second inequality holds by the definition of the $V_{h'}$, the third inequality holds because $R_{h'} \geq 2\ell+3K$, and the last inequality holds because $r_g, r_h \leq \ell$.

Now assume $h \in L^n = V(H^n-H^{n-1})$, and let $P$ be a $V_g$--$V_h$ path in $G$ of length $d_G(V_g,V_h)$. As $gh \notin E(H^{n+1})$, and the partition classes of nodes in $L^{n+1}$ are disjoint and not joined by an edge, $P$ meets either $G-G^{n+1}$ or it meets a bag $V_{h'}$ of some $h' \neq h \in L^n$ (see Figure~\ref{fig:ProofOf2}). In the former case, we obtain $d_G(V_g, V_h) \geq d_G(G-G^{n+1}, G^{n}) \geq 2\cdot\max\{r_g,r_h\} + 3K$ by the same argument as in \eqref{eq:a}. 
In the latter case, since the partition classes of nodes in $L^{n}$ are disjoint and not joined by an edge, $P$ has to meet either $G^{n-1}$, and we are done as before, or $P$ meets a bag~$V_{g'}$ of some $g' \neq g \in L^{n+1}$ (see Figure~\ref{fig:ProofOf2}). 
Then $d_G(V_g, V_h) \geq \max\{d_G(V_g,V_{g'}), d_G(V_h, V_{h'})\} \geq 2\cdot\max\{r_g, r_{g'}, r_h, r_{h'}\} + 3K$ by \ref{itm:FatK2t:DistanceBtwBags:Copy}, once we have proved that \ref{itm:FatK2t:DistanceBtwBags:Copy} holds for $g, g' \in L^{n+1}$. 

\begin{figure}[ht]
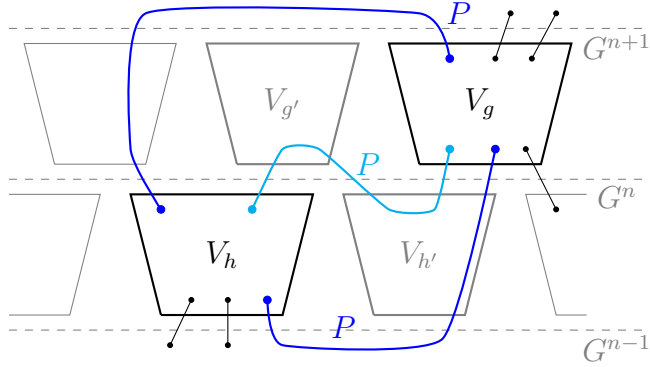

    \centering
    \include{Proof_of_2}
    \vspace{-2em}
    \caption{Depicted is the case where $g \in L^{n+1}$ and $h \in L^n$. The light blue path meets both $V_{h'}$ and $V_{g'}$. The dark blue paths meet either $G-G^{n+1}$ or $G^{n-1}$.}
    \label{fig:ProofOf2}
\end{figure}

Hence, it remains to consider the case where $g \neq h \in L^{n+1}$, i.e.\ $g = h_A$ and $h = h_B$ for some $A, B \in \cp$.
Let us first assume that $\cd_A \neq \cd_B$. 
If at least one of $\cd_A, \cd_B$ is of the form $\{C\}$ for some $C \in \cc(G-G^n)$, then one of $V_{h_A}$, $V_{h_B}$ is contained in $C$ and can be separated in $G$ from the other one by the partition class $V_x$ of the (unique) node $x\in L^n$ with $N_G(C) \subseteq V_x$.
Otherwise, $V_{h_A}, V_{h_B}$ are contained in distinct components $Z_A, Z_B$ of $G-G^{n-1}$. 
Hence, any $V_{h_A}$--$V_{h_B}$ path meets $G^{n-1}$, and so\\ $d_G(V_{h_A},V_{h_B}) \geq d_G(G-G^n, G^{n-1}) \geq 2\cdot\max\{r_{h_A}, r_{h_B}\} + 3K$ as in \eqref{eq:a}.

Thus, we may assume $\cd_A = \cd_B$. 
Let~$P$ be a $V_{h_A}$--$V_{h_B}$ path in~$G$ of length $d_G(V_{h_A}, V_{h_B})$.
If $P$ has a subpath that lies in $G-G^{n}$ and starts in $V_{h_A}$ and ends in $V_{h_{B'}}$ for some $B' \neq A \in \cp$ with $\cd_{B'} = \cd_A$, then
\begin{align*}
d_G(V_{h_A}, V_{h_B}) &\geq d_{G-G^n}(V_{h_A}, V_{h_{B'}}) \geq d_{G-G^n}(A,B') - R_{h_A} - R_{h_{B'}}\\
\intertext{and hence, by \eqref{distance} and the definition of $R_{h_A}, R_{h_B}$,}
d_G(V_{h_A}, V_{h_B}) &\geq (2L'_{\cd_A} + 4 \ell + 5K) - 2\cdot (L'_{\cd_A} + \ell + K)\\ 
&\geq 2 \ell + 3K \geq 2\cdot \max\{r_{h_A}, r_{h_B}\} + 3K.
\end{align*}
Otherwise, the path $P$ has a subpath that starts in $A$ and ends in $B'$ for some $B' \in \cp$. If $\cd_{B'} \neq \cd_A$, then we are done as in the previous case where $\cd_A \neq \cd_B$, so we may assume $\cd_{B'} =\cd_A$. 
Then by \eqref{distance1}
\[
d_G(V_{h_A}, V_{h_B}) \geq d_G(A, B') \geq 2L_{\cd_A} + 3K = 2\cdot\max\{r_{h_A}, r_{h_B}\} + 3K.
\]
This establishes \ref{itm:FatK2t:DistanceBtwBags:Copy} and hence concludes the proof. 
\qed

\section{The approximation algorithm} \label{sec:algo}

Note that our proof of \Lr{lem:ObtainingaGPwithNiceProperties} is constructive (and so are any lemmas it relies on), and therefore we will be able to turn it into an algorithm that approximates, to a constant factor, the optimal distortion \defi{$\alpha_t(G)$} of any embedding of a finite graph \g into a $K_{2,t}$-minor-free graph in polynomial time, thereby proving \Cr{cor:algo}:

\begin{proof}[Proof of \Cr{cor:algo}] 
Let $n:=|V(G)|$. For each $K=1,2,\ldots n$, our algorithm attempts to carry out the construction of $H$ and the $H$-partition of \g as in the proof of \Lr{lem:ObtainingaGPwithNiceProperties}, without knowing in advance whether \g has a $K$-fat $K_{2,t}$~minor. Note that the only occasions in that proof where we used the assumption that \g has no such minor were when invoking Lemmas~\ref{lem:NearComps:NumberAndRadius} and~\ref{lem:FatK2t:Components}. Thus, either the attempt will output such an $H$, or one of these calls to the aforementioned Lemmas will return a $K$-fat $K_{2,t}$ minor model in $G$, in which case we say that the attempt failed.
In the former case, where our algorithm constructs a graph $H$ and an $H$-partition of~$G$, it then checks whether $H$ is $K_{2,t}$-minor-free (which can be done in polynomial time~\cite{GMXIII}). If $H$ is $K_{2,t}$-minor-free, then we say that the attempt was \defi{successful}. If not, then the attempt failed, and invoking Corollary~\ref{cor:FirstLemma:PolyTimeAlgo} again returns a $K$-fat model of $K_{2,t}$ in $G$. 

Our algorithm returns the smallest value $K_{\min}$ of $K\leq n$ for which this procedure succeeds as an approximate value for $\alpha_t(G)$. Note that $K_{\min}$ exists since $G$ cannot have a $n$-fat $K_{2,0}$ minor. Along with $K_{\min}$, the algorithm can return a witness: we start with the graph $H$ and the embedding of \g into~$H$, defined by mapping each $v\in V(G)$ into its partition class $V_h\ni v$, and then modify $H$ and the embedding using the star trick mentioned before the statement of \Cr{cor:algo} to eliminate the additive error. 

We claim that $K_{\min}$ is within a constant factor of $\alpha_t(G)$. Indeed, our \Tr{thm:FatK2t} (and the remark thereafter) guarantees that the multiplicative distortion of \g into $H$, which is $K_{2,t}$-minor-free by definition, is at most $C\cdot K_{\min}$ for a universal constant $C$. If $K_{\min}>1$ then our procedure failed for $K=K_{\min}-1$, and therefore as mentioned above it will identify a $(K-1)$-fat $K_{2,t}$ minor model $M$ in \G. It is not hard to see that such a model implies that $\alpha_t(G)$ is at least $c\cdot (K_{\min}-1)$ for a small universal constant $c$ \cite[Proposition 3]{CDNRV} (the precise value of which depends on the convention chosen in the definition of multiplicative distortion). Our algorithm outputs $M$ as a witness for this lower bound on $\alpha_t(G)$. If on the other hand $K_{\min}=1$, then as above we deduce that $\alpha_t(G)\leq C$, and so we do not need a lower bound or a witness, as we can use the trivial bound $\alpha_t(G)\geq 1$. 
\end{proof}

Both the running time of our algorithm, and the approximation constant we obtained, increase with $t$. We do not know to what extent this is necessary. 

\medskip
Our \Cr{cor:algo}, along with analogous results of \cite{ABFNPT}, and remarks of \cite{CDNRV}, motivates the following problem related to the coarse Menger conjecture of \cite{AHJKW,GeoPapMin}.  Given a finite graph $G$, and $S,T \subset V(G)$, and $n\in \N$,  let \defi{$MM_n(G,S,T)$} denote the maximum $K\in \N$  such that there is an $n$-tuple of $S$--$T$ paths in $G$ pairwise at distance at least $K$.

\begin{problem}
    Is it true that for every $n\geq 2$, there are universal constants $C,c>1$, such that:
    \begin{enumerate}[label=\rm{(\roman*)}]
        \item \label{prob i} there is an efficient algorithm that, given $G,S,T$ as above, approximates $MM_n(G,S,T)$ up to a multiplicative factor of $C$; and
        \item \label{prob ii} approximating $MM_n(G,S,T)$ up to a multiplicative factor of $c$ is NP-hard.
    \end{enumerate}
\end{problem}

We remark that the results of \cite{AHJKW,GeoPapMin} that the coarse Menger conjecture is true for $n= 2$ imply that \ref{prob i} holds for $n=2$: the algorithm can output the smallest radius of a ball in \g separating $S$ from $T$. This trivially lower-bounds $MM_2(G,S,T)$, and the aforementioned result states that it is also an upper bound up to a universal constant $C$.

If we require the exact rather than an approximate value for $MM_n(G,S,T)$, the problem is NP-hard as proved by Balig\'acs and MacManus\cite{BalMcMmet}.

We do not know whether the analogue of \ref{prob ii} holds for $\alpha_t(G)$.

\bibliographystyle{plain}
\bibliography{collective}

\end{document}

%% file: defs.tex
\usepackage{amsthm,amssymb,amsmath,bbm,enumerate,graphicx,epsf,stmaryrd,accents}

\usepackage{authblk}

\newcommand{\comment}[1]{}
\newcommand{\COMMENT}[1]{}

\definecolor{darkgray}{rgb}{0.3,0.3,0.3}
\newcommand{\defi}[1]{{\color{darkgray}\emph{#1}}}

\comment{
	\begin{lemma}\label{}	
\end{lemma}
\begin{proof}

\end{proof}

\begin{theorem}\label{}
\end{theorem} 
\begin{proof} 	

\end{proof}

}

\newtheorem{proposition}{Proposition}[section]

\newtheorem{theorem}[proposition]{Theorem}
\newtheorem{corollary}[proposition]{Corollary}

\newtheorem{lemma}[proposition]{Lemma}

\newtheorem{conjecture}[proposition]{Conjecture}

\newtheorem{problem}[proposition]
{{Problem}}

\newtheorem{question}[proposition]{{Question}}

\newtheorem{examp}[proposition]{Example}

\newcommand{\FIG}{0}

\ifnum \NOTESON = 1 \newcommand{\note}[1]{ 

\hspace*{-30pt}
	{\color{blue}  NOTE: \color{Turquoise}{\small  \tt \begin{minipage}[c]{1.1\textwidth}  #1 \end{minipage} \ignorespacesafterend }} 
	
	}
\else \newcommand{\note}[1]{} \fi

\newcommand{\afsubm}[1]{ \ifnum \Debug = 1 {\mymargin{#1}}
\fi} 

\ifnum \Debug = 1 
\else  \fi

\ifnum \FIG = 1 
\else  \fi

\ifnum \FIG = 1 
\else  \fi

\ifnum \Debug = 1 \usepackage[notref,notcite]{showkeys}
\fi

\ifnum \COLORON = 0 \renewcommand{\color}[1]{}
\fi

\newcommand{\N}{\ensuremath{\mathbb N}}
\newcommand{\R}{\ensuremath{\mathbb R}}

\newcommand{\ca}{\ensuremath{\mathcal A}}
\newcommand{\cb}{\ensuremath{\mathcal B}}
\newcommand{\cc}{\ensuremath{\mathcal C}}
\newcommand{\cd}{\ensuremath{\mathcal D}}
\newcommand{\ce}{\ensuremath{\mathcal E}}

\newcommand{\ch}{\ensuremath{\mathcal H}}

\newcommand{\cj}{\ensuremath{\mathcal J}}

\newcommand{\cp}{\ensuremath{\mathcal P}}
\newcommand{\cq}{\ensuremath{\mathcal Q}}

\newcommand{\cgr}{\ensuremath{\mathcal R}}

\newcommand{\cu}{\ensuremath{\mathcal U}}

\makeatletter
\DeclareRobustCommand{\cev}[1]{%
  \mathpalette\do@cev{#1}%
}
\newcommand{\do@cev}[2]{%
  \fix@cev{#1}{+}%
  \reflectbox{$\m@th#1\vec{\reflectbox{$\fix@cev{#1}{-}\m@th#1#2\fix@cev{#1}{+}$}}$}%
  \fix@cev{#1}{-}%
}
\newcommand{\fix@cev}[2]{%
  \ifx#1\displaystyle
    \mkern#23mu
  \else
    \ifx#1\textstyle
      \mkern#23mu
    \else
      \ifx#1\scriptstyle
        \mkern#22mu
      \else
        \mkern#22mu
      \fi
    \fi
  \fi
}

\makeatother

\newcommand{\g}{\ensuremath{G\ }}
\newcommand{\G}{\ensuremath{G}}

\newcommand{\Lr}[1]{Lemma~\ref{#1}}

\newcommand{\Tr}[1]{Theorem~\ref{#1}}

\newcommand{\Sr}[1]{Section~\ref{#1}}

\newcommand{\Cr}[1]{Corollary~\ref{#1}}
\newcommand{\Cnr}[1]{Con\-jecture~\ref{#1}}

\newcommand{\fe}{for every}

\newcommand{\st}{such that}

\newcommand{\labtequtag}[3]{%
 \begin{equation} \label{#1} 	\begin{minipage}[c]{0.9\textwidth}  #2 \end{minipage} \ignorespacesafterend \tag{#3} \end{equation} }

\newcommand{\mymargin}[1]{
 \ifnum \Debug = 1
  \marginpar{%
    \begin{minipage}{\marginparwidth}\small%
      \begin{flushleft}%
        {\color{blue}#1}%
      \end{flushleft}%
   \end{minipage}%
  }%
 \fi
}%

\newcommand{\extras}[1]{
 \ifnum \Debug = 1
\section{Extras} #1
 \fi
}%

\newcommand{\mySection}[2]{}



%% file: Lemma_FatXMinor.tex
\begin{tikzpicture}[scale=0.4]

        \draw[gray,dotted,thick] (0,-0.15)--(21,-0.15);
        \draw[gray,dotted,thick] (0,6.15)--(21,6.15);

        \draw[thick,gray] (0,0)--(1,0)--(3,6)--(0,6);
        \draw[cyan,thick] (0,-0.3)--(0.85,-0.3)--(1.5,-2)--(0,-2);

        \draw[thick] (7,0)--(14,0)--(16,6)--(5,6)--(7,0);
        \draw[blue,thick] (7.15,0.15)--(13.85,0.15)--(14.85,3)--(6.15,3)--(7.15,0.15);
        \draw[blue,thick] (7.15,-0.3)--(13.85,-0.3)--(14.5,-2)--(6.5,-2)--(7.15,-0.3);

        \draw[gray,thick] (7.5,8)--(8,6.3)--(13,6.3)--(13.5,8);
        \draw[cyan,thick] (8.15,5.85)--(12.85,5.85)--(13.3,5)--(7.7,5)--(8.15,5.85);

        \draw[Green,thick] (7.3,0.3)--(13.7,0.3)--(13.9,0.8)--(7.1,0.8)--(7.3,0.3);

        \draw[gray,thick,<->] (2,2)--(6,2);
        \draw[gray,thick,<->] (1.7,-1.5)--(6.4,-1.5);
        \draw[gray,thick,<->] (10.5,3.2)--(10.5,4.8);
        \draw[thick,<->] (16.9,0)--(16.9,6);
        \draw[blue,thick,<->] (15.3,-0.3)--(15.3,-2);

        \draw[gray] (20.3,5.6) node {$G^{n}$};
        \draw[gray] (20,-0.7) node {$G^{n-1}$};

        \draw[gray] (0.9,3) node {$V_{g'}$};
        \draw[cyan] (0.45,-1.1) node {$V_{g'}^\downarrow$};
        \draw[blue] (10.5,-1.1) node {$V_h^\downarrow$};
        \draw[blue] (10.5,1.9)  node {$V_h^{\uparrow}$};
        \draw[Green] (13.1,1.35) node {\scriptsize{$\partial_h^{\downarrow}$}};
        \draw[cyan] (13.8,5.4) node {\footnotesize{$V_g^\downarrow$}};
        \draw[gray] (10.5,7.3) node {$V_g$};
        \draw (15.9,3) node {$V_h$};

        \draw (17.7,3) node {$R_h$};
        \draw[blue] (16,-1.2) node {$r_h$};

        \draw[gray] (11.8,4) node {\footnotesize{$\geq 3K$}};
        \draw[gray] (4,2.5) node {\footnotesize{$\geq 3K$}};
        \draw[gray] (4,-1) node {\footnotesize{$\geq 3K$}};
        
    \end{tikzpicture}

%% file: Fat_K2t_Fat_XMinor.tex
    \begin{tikzpicture}[scale=0.45]

    \draw[gray,dotted,thick] (3,-0.15)--(18,-0.15);
    \draw[gray,dotted,thick] (3,6.15)--(18,6.15);

    \draw[thick] (7,0)--(14,0)--(16,6)--(5,6)--(7,0);
    \draw[blue,thick] (7.15,-0.3)--(13.85,-0.3)--(14.5,-2)--(6.5,-2)--(7.15,-0.3);
    \draw[blue,thick] (7.15,0.15)--(13.85,0.15)--(14.5,2)--(6.5,2)--(7.15,0.15);

    \draw[gray,thick] (8.8,7.7)--(9.3,6.3)--(11.7,6.3)--(12.2,7.7);
    \draw[gray,thick] (5,7.7)--(5.5,6.3)--(7.5,6.3)--(8,7.7);
    \draw[gray,thick] (13,7.7)--(13.5,6.3)--(15.5,6.3)--(16,7.7);

    \draw[Green,thick] (5.65,5.85)--(7.35,5.85)--(7.55,5)--(5.45,5)--(5.65,5.85);
    \draw[Fuchsia,thick] (9.45,5.85)--(11.55,5.85)--(11.75,5)--(9.25,5)--(9.45,5.85);
    \draw[cyan,thick] (13.65,5.85)--(15.35,5.85)--(15.55,5)--(13.45,5)--(13.65,5.85);

    \draw[Green,thick] (5.65,6.45)--(7.35,6.45)--(7.55,7)--(5.45,7)--(5.65,6.45);
    \draw[Fuchsia,thick] (9.45,6.45)--(11.55,6.45)--(11.75,7)--(9.25,7)--(9.45,6.45);
    \draw[cyan,thick] (13.65,6.45)--(15.35,6.45)--(15.55,7)--(13.45,7)--(13.65,6.45);

    \draw[blue,thick] (8,0.5)--(7.5,-0.75);
    \draw[blue,thick] (9,0.5)--(8.8,-0.75);
    \draw[blue,thick] (10,0.5)--(10,-0.75);
    \draw[blue,thick] (11,0.5)--(11,-0.75);
    \draw[blue,thick] (12,0.5)--(12.2,-0.75);
    \draw[blue,thick] (13,0.5)--(13.5,-0.75);

    \draw[fill,orange] (6,5.3) circle (.1);
    \draw[thick,orange] (6,5.3)--(6.6,3.5);
    \draw[thick,blue] (6.6,3.5)--(7.2,1.7);
    \draw[fill,orange] (6.6,3.5) circle (.1);
    \draw[fill,blue] (7.2,1.7) circle (.1);

    \draw[fill,orange] (15,5.3) circle (.1);
    \draw[thick,orange] (15,5.3)--(14.4,3.5);
    \draw[thick,blue] (14.4,3.5)--(13.8,1.7);
    \draw[fill,orange] (14.4,3.5) circle (.1);
    \draw[fill,blue] (13.8,1.7) circle (.1);

    \draw[gray] (17.3,5.6) node {$G^{i_x}$};
    \draw[gray] (17,-0.7) node {$G^{i_x-1}$};

    \draw[blue] (15.3,0.7) node {$U_x$};
    \draw (16.4,2.9) node {$V_x$};
    \draw[blue] (10.5,1.2) node {$V_x^\uparrow$};
    \draw[blue] (10.5,-1.2) node {$V_x^\downarrow$};

    \draw[orange] (7.2,4.2) node {\footnotesize{$E_{xy}$}};
    \draw[orange] (13.8,4.2) node {\footnotesize{$E_{xy'}$}};

    \draw[blue] (7.7,2.7) node {\footnotesize{$T_{xy}$}};
    \draw[blue] (13.3,2.7) node {\footnotesize{$T_{xy'}$}};

    \draw[gray] (4.5,7) node {\footnotesize{$V_{y}$}};
    \draw[gray] (12.6,7) node {\footnotesize{$V_{z}$}};
    \draw[gray] (16.8,7) node {\footnotesize{$V_{y'}$}};
    
    \end{tikzpicture}

%% file: FatK2t_Lemma1_Proof.tex
\begin{tikzpicture}[scale=0.3]
    
        \draw[thick] (0,0)--(9,0)--(9,7)--(0,7)--(0,0);
        \draw[Green,thick] (2,2)--(7,2)--(7,5)--(2,5)--(2,2);

        \draw[blue,thick] (12,0)--(21,0)--(21,7)--(12,7)--(12,0);
        \draw[thick] (14,2)--(19,2)--(19,5)--(14,5)--(14,2);

        \draw [blue, thick] plot [smooth, tension=0.4] coordinates {(1.8,15) (2.2,13) (2.9,11) (3.8,10) (4.5,9.8) (5.2,10) (6.1,11) (6.8,13) (7.2,15)};

        \draw [blue, thick] plot [smooth, tension=0.4] coordinates {(7.8,15) (8.2,13) (8.9,11) (9.8,10) (10.5,9.8) (11.2,10) (12.1,11) (12.8,13) (13.2,15)};

        \draw [blue, thick] plot [smooth, tension=0.4] coordinates {(13.8,15) (14.2,13) (14.9,11) (15.8,10) (16.5,9.8) (17.2,10) (18.1,11) (18.8,13) (19.2,15)};

        \draw[orange,thick] (3.85,10.85)--(2.7,4.3);
        \draw[fill,orange] (3.85,10.85) circle (.2);
        \draw[fill,orange] (2.7,4.3) circle (.2);

        \draw[orange,thick] (9.85,10.85)--(5,4.3);
        \draw[fill,orange] (9.85,10.85) circle (.2);
        \draw[fill,orange] (5,4.3) circle (.2);

        \draw[orange,thick] (15.85,10.85)--(6.3,4.3);
        \draw[fill,orange] (15.85,10.85) circle (.2);
        \draw[fill,orange] (6.3,4.3) circle (.2);

        \draw[blue,thick] (5.15,10.85)--(13,6.3);
        \draw[fill,blue] (5.15,10.85) circle (.2);
        \draw[fill,blue] (13,6.3) circle (.2);

        \draw[blue,thick] (11.15,10.85)--(15,6.3);
        \draw[fill,blue] (11.15,10.85) circle (.2);
        \draw[fill,blue] (15,6.3) circle (.2);

        \draw[blue,thick] (17.15,10.85)--(18,6.3);
        \draw[fill,blue] (17.15,10.85) circle (.2);
        \draw[fill,blue] (18,6.3) circle (.2);

        \draw[Green] (4.5,3) node {$X_i=V_1$};
        \draw (16.5,3.3) node {$X_j$};
        \draw (4.5,0.8) node {\footnotesize $B_G(X_i, K-1)$};
        \draw[blue] (16.5,0.8) node {\footnotesize $B_G(X_j, K-1)$};

        \draw[blue] (20,9) node {\large $V_2$};

        \draw[blue] (10.5,14) node {$C_w$};
        \draw[above,orange] (10.3,10.85) node {\scriptsize $w \in W$};
        \draw[orange] (6.5,8) node {$Q_w$};

    \end{tikzpicture}

%% file: Partition_R.tex
\begin{tikzpicture}[scale=0.35,auto=left]

\definecolor{dgreen}{rgb}{0,0.7,0}

\draw[gray,dashed] (0,0) -- (34,0);
\draw[gray,dashed] (0,4) -- (34,4);
\draw[gray,dashed] (0,5.6) -- (34,5.6);

\draw [thick] plot [smooth, tension=0.2] coordinates {(0,10) (1,0.6) (8,0.6) (9,10)};
\draw [thick] plot [smooth, tension=0.2] coordinates {(10,10) (11,0.6) (22,0.6) (23,10)};
\draw [thick] plot [smooth, tension=0.2] coordinates {(24,10) (25,0.6) (30,0.6) (31,10)};

\draw [thick,dgreen,dashed] plot [smooth, tension=0.2] coordinates {(0.3,10) (0.8,4.5) (2.2,4.5) (2.5,10)};
\draw [thick,dgreen,dashed] plot [smooth, tension=0.2] coordinates {(2.8,10) (3.1,4.5) (6,4.5) (6.2,10)};
\draw [thick,dgreen,dashed] plot [smooth, tension=0.2] coordinates {(6.5,10) (6.8,4.5) (8.2,4.5) (8.7,10)};

\draw [gray,line width=0.5] plot [smooth, tension=0.2] coordinates {(0.5,10) (1,4.7) (2,4.7) (2.3,10)};
\draw [cyan,line width=0.5] plot [smooth, tension=0.2] coordinates {(3,10) (3.3,4.7) (4.3,4.7) (4.4,10)};
\draw [cyan,line width=0.5] plot [smooth, tension=0.2] coordinates {(4.6,10) (4.7,4.7) (5.8,4.7) (6,10)};
\draw [gray,line width=0.5] plot [smooth, tension=0.2] coordinates {(6.7,10) (7,4.7) (8,4.7) (8.5,10)};

\draw [thick,dgreen,dashed] plot [smooth, tension=0.12] coordinates {(10.3,10) (10.9,4.4) (19.9,4.4) (20.2,10)};
\draw [thick,dashed,dgreen] plot [smooth, tension=0.2] coordinates {(20.5,10) (20.8,4.5) (22.2,4.5) (22.7,10)};

\draw [blue,thick] plot [smooth, tension=0.2] coordinates {(10.6,10) (11.1,4.7) (13.9,4.7) (14.15,10)};
\draw [blue,thick] plot [smooth, tension=0.2] coordinates {(14.4,10) (14.7,4.7) (16.8,4.7) (17.05,10)};
\draw [blue,thick] plot [smooth, tension=0.2] coordinates {(17.3,10) (17.6,4.7) (19.7,4.7) (19.9,10)};
\draw [gray,line width=0.5] plot [smooth, tension=0.2] coordinates {(20.7,10) (21,4.7) (22,4.7) (22.5,10)};

\draw[thick,dgreen,dashed] plot [smooth, tension=0.2] coordinates {(24.3,10) (24.9,4.5) (30.1,4.5) (30.7,10)};
\draw[cyan,line width=0.5] plot [smooth, tension=0.2] coordinates {(24.5,10) (25.1,4.7) (27,4.7) (27.3,10)};
\draw[cyan,line width=0.5] plot [smooth, tension=0.2] coordinates {(27.6,10) (27.9,4.7) (29.9,4.7) (30.5,10)};

\draw[gray,thick] (1.3,0.6) -- (2.9,0.6) -- (3.2,3.7) -- (0.9,3.7) -- (1.3,0.6);
\draw[gray,thick] (3.9,0.6) -- (5.9,0.6) -- (6.2,3.7) -- (3.6,3.7) -- (3.9,0.6);
\draw[gray,thick] (7,0.6) -- (12,0.6) -- (12.3,3.7) -- (6.7,3.7) -- (7,0.6);
\draw[gray,thick] (12.7,3.7) -- (13,0.6) -- (15.5,0.6) -- (15.8,3.7) -- (12.7,3.7);
\draw[gray,thick] (16.2,3.7) -- (16.5,0.6) -- (19,0.6) -- (19.3,3.7) -- (16.2,3.7);
\draw[gray,thick] (19.7,3.7) -- (20,0.6) -- (21.8,0.6) -- (22.3,3.7) -- (19.7,3.7);
\draw[gray,thick] (25,3.7) -- (25.3,0.6) -- (27.4,0.6) -- (27.7,3.7) -- (25,3.7);
\draw[gray,thick] (28.1,3.7) -- (28.4,0.6) -- (29.7,0.6) -- (30.1,3.7) -- (28.1,3.7);

\draw[magenta,thick] (11.3,4.8) -- (12,4.8) -- (12,5.4) -- (11.3,5.4) -- (11.3,4.8);
\draw[magenta,thick] (12.2,4.8) -- (13,4.8) -- (13,5.4) -- (12.2,5.4) -- (12.2,4.8);
\draw[magenta,thick] (13.2,4.8) -- (13.8,4.8) -- (13.8,5.4) -- (13.2,5.4) -- (13.2,4.8);
\draw[magenta,thick] (14.9,4.8) -- (15.8,4.8) -- (15.8,5.4) -- (14.9,5.4) -- (14.9,4.8);
\draw[magenta,thick] (16,4.8) -- (16.65,4.8) -- (16.65,5.4) -- (16,5.4) -- (16,4.8);
\draw[magenta,thick] (17.8,4.8) -- (18.5,4.8) -- (18.5,5.4) -- (17.8,5.4) -- (17.8,4.8);
\draw[magenta,thick] (18.7,4.8) -- (19.5,4.8) -- (19.5,5.4) -- (18.7,5.4) -- (18.7,4.8);

\draw[gray,line width=0.5] (1.3,4.9) -- (1.2,3.3);
\draw[gray,line width=0.5] (1.5,4.9) -- (1.5,3.3);
\draw[gray,line width=0.5] (1.7,4.9) -- (2.1,3.3);

\draw[cyan,line width=0.5] (3.5,4.9) -- (2.8,3.3);
\draw[cyan,line width=0.5] (4,4.9) -- (4.3,3.3);

\draw[cyan,line width=0.5] (5,4.9) -- (4.8,3.3);
\draw[cyan,line width=0.5] (5.5,4.9) -- (7,3.3);

\draw[gray,line width=0.5] (7.4,4.9) -- (7.4,3.3);
\draw[gray,line width=0.5] (7.8,4.9) -- (8,3.3);

\draw[blue,line width=0.7] (11.7,5) -- (11.5,3.3);
\draw[blue,line width=0.7] (12.6,5) -- (13.2,3.3);
\draw[blue,line width=0.7] (13.5,5) -- (14,3.3);

\draw[blue,line width=0.7] (15.2,5) -- (14.8,3.3);
\draw[blue,line width=0.7] (16.2,5) -- (16.8,3.3);

\draw[blue,line width=0.7] (18,5) -- (15.2,3.3);
\draw[blue,line width=0.7] (18.2,5) -- (17.5,3.3);
\draw[blue,line width=0.7] (19.2,5) -- (20,3.3);

\draw[gray,line width=0.5] (21.2,4.9) -- (21,3.3);
\draw[gray,line width=0.5] (21.6,4.9) -- (21.8,3.3);

\draw[cyan,line width=0.5] (25.6,4.9) -- (25.4,3.3);
\draw[cyan,line width=0.5] (26,4.9) -- (26.1,3.3);
\draw[cyan,line width=0.5] (26.6,4.9) -- (28.5,3.3);

\draw[cyan,line width=0.5] (28.3,4.9) -- (27,3.3);
\draw[cyan,line width=0.5] (28.8,4.9) -- (29,3.3);
\draw[cyan,line width=0.5] (29.3,4.9) -- (29.7,3.3);

\draw[gray] (32.7,-0.6) node [] {$G^{n-1}$};
\draw[gray] (33.2,3.4) node [] {$G^{n}$};
\draw[gray] (32.5,4.7) node [] {$N_G(G^n)$};
\draw[dgreen] (32,8) node [] {$\mathfrak{R}$};
\draw (23.1,1.8) node [] {$Z$};
\draw[blue] (15.6,9.5) node [] {$\cd_Z$};
\draw[magenta] (15.8,6.5) node [] {$\cb_{\cd_Z}$};
\draw[blue] (11.2,9.5) node [] {\footnotesize $C$};
\draw[gray] (11.7,6.1) node [] {\scriptsize $\partial C$};
\draw[magenta] (13.4,6.1) node [] {\scriptsize $B_i^C$};

\end{tikzpicture}

%% file: MergingLemma1.tex
\begin{tikzpicture}[scale=0.35,auto=left]

\definecolor{dgreen}{rgb}{0,0.7,0}

\draw[gray,dashed] (-2.55,0) -- (31,0);
\draw[gray,dashed] (-2.55,1.8) -- (31,1.8);

\draw[dgreen,thick] (0.5,0.3) -- (15,0.3) -- (15,1.5) -- (0.5,1.5) -- (0.5,0.3);
\draw[magenta,thick] (0.8,0.55) -- (4,0.55) -- (4,1.25) -- (0.8,1.25) -- (0.8,0.55);
\draw[magenta,thick] (7.2,0.55) -- (10,0.55) -- (10,1.25) -- (7.2,1.25) -- (7.2,0.55);
\draw[magenta,thick] (12,0.55) -- (14.7,0.55) -- (14.7,1.25) -- (12,1.25) -- (12,0.55);

\draw[dgreen,thick] (20.2,0.3) -- (27.8,0.3) -- (27.8,1.5) -- (20.2,1.5) -- (20.2,0.3);
\draw[magenta,thick] (20.5,0.55) -- (23,0.55) -- (23,1.25) -- (20.5,1.25) -- (20.5,0.55);
\draw[magenta,thick] (25.5,0.55) -- (27.5,0.55) -- (27.5,1.25) -- (25.5,1.25) -- (25.5,0.55);

\draw[gray,thick] (0.7,0.55) -- (0.7,1.25) -- (0,3.5) -- (4.8,3.5) -- (4.1,1.25) -- (4.1,0.55) -- (4.8,-1.5) -- (0,-1.5) -- (0.7,0.55);
\draw[gray,thick] (7.1,0.55) -- (7.1,1.25) -- (6.4,3.5) -- (10.8,3.5) -- (10.1,1.25) -- (10.1,0.55) -- (10.8,-1.5) -- (6.4,-1.5) -- (7.1,0.55);
\draw[gray,thick] (11.9,0.55) -- (11.9,1.25) -- (11.2,3.5) -- (15.5,3.5) -- (14.8,1.25) -- (14.8,0.55) -- (15.5,-1.5) -- (11.2,-1.5) -- (11.9,0.55);

\draw[gray,thick] (20.4,0.55) -- (20.4,1.25) -- (19.7,3.5) -- (23.7,3.5) -- (23.1,1.25) -- (23.1,0.55) -- (23.7,-1.5) -- (19.7,-1.5) -- (20.4,0.55);
\draw[gray,thick] (25.4,0.55) -- (25.4,1.25) -- (24.7,3.5) -- (28.3,3.5) -- (27.6,1.25) -- (27.6,0.55) -- (28.3,-1.5) -- (24.7,-1.5) -- (25.4,0.55);

\draw[thick] (0.4,0.3) -- (0.4,1.5) -- (-0.7,5) -- (16.2,5) -- (15.1,1.5) -- (15.1,0.3) -- (16.2,-3) -- (-0.7,-3) -- (0.4,0.3);

\draw[thick] (20.1,0.3) -- (20.1,1.5) -- (19,5) -- (29,5) -- (27.9,1.5) -- (27.9,0.3) -- (29,-3) -- (19,-3) -- (20.1,0.3);

\draw[gray,<->] (1.8,2) -- (1.8,3.3);
\draw[line width=0.6pt,<->] (-1.2,0.95) -- (-1.2,5);
\draw[line width=0.6pt,<->] (-1.2,0.85) -- (-1.2,-3);
\draw[line width=0.6pt,<->] (16.4,4.8) -- (18.8,4.8);
\draw[gray,<->] (4.9,3.7) -- (6.3,3.7);
\draw[line width=0.6pt,dgreen,<->] (0.5,-3.4) -- (15,-3.4);
\draw[line width=0.6pt,dgreen,<->] (20.2,-3.4) -- (27.8,-3.4);

\draw[gray] (30.3,-0.6) node [] {$G^{n}$};
\draw[gray] (29.8,0.9) node [] {\small $N(G^n)$};

\draw[gray] (2.7,2.6) node [] {\scriptsize $\frac{3K}{2}$};
\draw (17.6,4.1) node [] {\footnotesize $\geq 3K$};

\draw[gray] (5.6,4.1) node [] {\tiny $< 3K$};

\draw (-1.95,3) node [] {$L_\cd$};
\draw (-1.95,-1.3) node [] {$L_\cd$};

\draw[dgreen] (18.7,0.85) node [] {$\cq_\cd$};
\draw[magenta] (16.5,0.85) node [] {$\cb_\cd$};

\draw[dgreen] (8,-4) node [] {\small $\leq D_1 (= D_1(t,K))$};
\draw[dgreen] (24,-4) node [] {\small $\leq D_1$};

\end{tikzpicture}

%% file: MergingLemma2.tex
\begin{tikzpicture}[scale=0.35,auto=left]

\definecolor{dgreen}{rgb}{0,0.7,0}

\draw[gray,dashed] (-2.55,0) -- (31,0);
\draw[gray,dashed] (-2.55,1.8) -- (31,1.8);

\draw[blue,thick] (0.5,0.3) -- (15,0.3) -- (15,1.5) -- (0.5,1.5) -- (0.5,0.3);
\draw[dgreen,thick] (0.8,0.55) -- (4,0.55) -- (4,1.25) -- (0.8,1.25) -- (0.8,0.55);
\draw[dgreen,thick] (7.2,0.55) -- (10,0.55) -- (10,1.25) -- (7.2,1.25) -- (7.2,0.55);
\draw[dgreen,thick] (12,0.55) -- (14.7,0.55) -- (14.7,1.25) -- (12,1.25) -- (12,0.55);

\draw[blue,thick] (20.2,0.3) -- (27.8,0.3) -- (27.8,1.5) -- (20.2,1.5) -- (20.2,0.3);
\draw[dgreen,thick] (20.5,0.55) -- (23,0.55) -- (23,1.25) -- (20.5,1.25) -- (20.5,0.55);
\draw[dgreen,thick] (25.5,0.55) -- (27.5,0.55) -- (27.5,1.25) -- (25.5,1.25) -- (25.5,0.55);


\draw[gray,thick] (0.7,0.55) -- (0.7,1.25) -- (0.4,3) -- (4.4,3) -- (4.1,1.25) -- (4.1,0.55) -- (4.4,-1) -- (0.4,-1) -- (0.7,0.55);
\draw[gray,dashed] (0.4,3) -- (0.2,4.1) -- (4.6,4.1) -- (4.4,3);

\draw[gray,thick] (7.1,0.55) -- (7.1,1.25) -- (6.8,3) -- (10.4,3) -- (10.1,1.25) -- (10.1,0.55) -- (10.4,-1) -- (6.8,-1) -- (7.1,0.55);
\draw[gray,dashed] (6.8,3) -- (6.6,4.1) -- (10.6,4.1) -- (10.4,3);

\draw[gray,thick] (11.9,0.55) -- (11.9,1.25) -- (11.6,3) -- (15.1,3) -- (14.8,1.25) -- (14.8,0.55) -- (15.1,-1) -- (11.6,-1) -- (11.9,0.55);
\draw[gray,dashed] (11.6,3) -- (11.4,4.1) -- (15.3,4.1) -- (15.1,3);

\draw[gray,thick] (20.4,0.55) -- (20.4,1.25) -- (20.1,3) -- (23.4,3) -- (23.1,1.25) -- (23.1,0.55) -- (23.4,-1) -- (20.1,-1) -- (20.4,0.55);
\draw[gray,dashed] (20.1,3) -- (19.9,4.1) -- (23.6,4.1) -- (23.4,3);

\draw[gray,thick] (25.4,0.55) -- (25.4,1.25) -- (25,3) -- (27.9,3) -- (27.6,1.25) -- (27.6,0.55) -- (27.9,-1) -- (25,-1) -- (25.4,0.55);
\draw[gray,dashed] (25,3) -- (24.8,4.1) -- (28.1,4.1) -- (27.9,3);

\draw[thick] (0.4,0.3) -- (0.4,1.5) -- (-0.3,5.5) -- (15.8,5.5) -- (15.1,1.5) -- (15.1,0.3) -- (15.4,-1.2) -- (0.1,-1.2) -- (0.4,0.3);
\draw[thick,dashed] (-0.3,5.5) -- (-0.7,7.5) -- (16.2,7.5) -- (15.8,5.5);

\draw[thick] (20.1,0.3) -- (20.1,1.5) -- (19.4,5.5) -- (28.6,5.5) -- (27.9,1.5) -- (27.9,0.3) -- (28.2,-1.2) -- (19.8,-1.2) -- (20.1,0.3);
\draw[thick,dashed] (19.4,5.5) -- (19,7.5) -- (29,7.5) -- (28.6,5.5);


\draw[gray,<->] (4.7,4.3)--(6.5,4.3);
\draw[gray,<->] (1.7,3.1) -- (1.7,4);
\draw[line width=0.6pt,<->] (-1.2,0.95) -- (-1.2,5.5);
\draw[line width=0.6pt,<->] (-1.2,0.85) -- (-1.2,-1.2);
\draw[line width=0.6pt,<->] (16.4,7.4) -- (18.8,7.4);
\draw[line width=0.6pt,gray,<->] (4.7,2.9) -- (6.5,2.9);
\draw[line width=0.6pt,blue,<->] (0.5,-1.6) -- (15,-1.6);
\draw[line width=0.6pt,blue,<->] (20.2,-1.6) -- (27.8,-1.6);
\draw[line width=0.6pt,<->] (6.7,5.6) -- (6.7,7.4);

\draw[gray] (30.3,-0.6) node [] {$G^{n}$};
\draw[gray] (29.8,0.9) node [] {\small $N(G^n)$};

\draw[gray] (5.6,4.8) node [] {\tiny $< 3K$};
\draw[gray] (3,3.55) node [] {\scriptsize $\ell+K$};
\draw (8.4,6.5) node [] {$\ell+K$};

\draw (17.6,6.7) node [] {\footnotesize $\geq 3K$};

\draw[gray] (5.6,2.3) node [] {\scriptsize $\geq 3K$};

\draw (-1.95,3.3) node [] {$L'_\cd$};
\draw (-1.95,-0.2) node [] {$L_\cd$};

\draw[blue] (18.7,0.85) node [] {$\cp_\cd$};
\draw[dgreen] (16.5,0.85) node [] {$\cq_\cd$};

\draw[blue] (8,-2.2) node [] {\small $\leq R_0(t,K)$};
\draw[blue] (24,-2.2) node [] {\small $\leq R_0(t,K)$};

\end{tikzpicture}

%% file: Proof_of_2.tex
\begin{tikzpicture}[scale=0.4]

\draw[gray,dashed] (0,0) -- (21,0);
\draw[gray,dashed] (0,5) -- (21,5);
\draw[gray,dashed] (0,10) -- (21,10);

\draw[gray] (0,0.5) -- (2,0.5) -- (3,4.5) -- (0,4.5);
\draw[thick] (5,0.5) -- (9,0.5) -- (10,4.5) -- (4,4.5) -- (5,0.5);
\draw[gray,thick] (12,0.5) -- (15,0.5) -- (16,4.5) -- (11,4.5) -- (12,0.5);
\draw[gray] (19,0.5) -- (18,0.5) -- (17,4.5) -- (19,4.5);

\draw[gray] (1.5,5.5) -- (4.5,5.5) -- (5.5,9.5) -- (0.5,9.5) -- (1.5,5.5);
\draw[gray,thick] (7.5,5.5) -- (10.5,5.5) -- (11.5,9.5) -- (6.5,9.5) -- (7.5,5.5);
\draw[thick] (13.5,5.5) -- (17.5,5.5) -- (18.5,9.5) -- (12.5,9.5) -- (13.5,5.5);

\draw [cyan, thick] plot [smooth, tension=0.3] coordinates {(14.5,6) (14,4) (12.5,4) (10.2,6) (9,6) (8,4)};
\draw[fill,cyan] (14.5,6) circle (.125);
\draw[fill,cyan] (8,4) circle (.125);

\draw [blue, thick] plot [smooth, tension=0.3] coordinates {(14.5,9) (13.5,10.5) (4.5,10.5) (4,6) (5,4)};
\draw[fill,blue] (14.5,9) circle (.125);
\draw[fill,blue] (5,4) circle (.125);

\draw [blue, thick] plot [smooth, tension=0.3] coordinates {(16,6) (14.25,-0.3) (9,-0.5) (8.5,1)};
\draw[fill,blue] (16,6) circle (.125);
\draw[fill,blue] (8.5,1) circle (.125);

\draw (16,9) -- (16.5,10.5);
\draw[fill,black] (16,9) circle (.08);
\draw[fill,black] (16.5,10.5) circle (.08);

\draw (17.2,9) -- (18,10.5);
\draw[fill,black] (17.2,9) circle (.08);
\draw[fill,black] (18,10.5) circle (.08);

\draw (17,6) -- (18,4);
\draw[fill,black] (17,6) circle (.08);
\draw[fill,black] (18,4) circle (.08);

\draw (6,1) -- (5.3,-0.5);
\draw[fill,black] (6,1) circle (.08);
\draw[fill,black] (5.3,-0.5) circle (.08);

\draw (7.2,1) -- (7.2,-0.5);
\draw[fill,black] (7.2,1) circle (.08);
\draw[fill,black] (7.2,-0.5) circle (.08);

\draw (7,2.5) node {\large $V_h$};
\draw (15.5,7.5) node {\large $V_g$};
\draw[gray] (13.5,2.5) node {\large $V_{h'}$};
\draw[gray] (9,7.5) node {\large $V_{g'}$};

\draw[cyan] (11.8,5.5) node {\large $P$};
\draw[blue] (14.8,10.5) node {\large $P$};
\draw[blue] (11,0.1) node {\large $P$};

\draw[gray] (20,-0.6) node {\large $G^{n-1}$};
\draw[gray] (20,4.4) node {\large $G^{n}$};
\draw[gray] (20,9.4) node {\large $G^{n+1}$};

\end{tikzpicture}

    

%% file: main.bbl
\begin{thebibliography}{10}

\bibitem{ABFNPT}
R.~Agarwala, V.~Bafna, M.~Farach, B.~Narayanan, M.~Paterson, and M.~Thorup.
\newblock On the approximability of numerical taxonomy (fitting distances by tree metrics).
\newblock {\em {SIAM J.\ Comput.}}, 28:1073--1085, 1999.

\bibitem{ADWeakCounterex}
S.~Albrechtsen and J.~Davies.
\newblock Counterexample to the conjectured coarse grid theorem.
\newblock arXiv:2508.15342.

\bibitem{radialpathwidth}
S.~Albrechtsen, R.~Diestel, A.-K. Elm, E.~Fluck, R.~W. Jacobs, P.~Knappe, and P.~Wollan.
\newblock A structural duality for path-decompositions into parts of small radius.
\newblock arXiv:2307.08497.

\bibitem{ADGSmallCounterexamples}
S.~Albrechtsen, M.~Distel, and A.~Georgakopoulos.
\newblock Small counterexamples to the fat minor conjecture.
\newblock In preparation.

\bibitem{AHJKW}
S.~Albrechtsen, T.~Huynh, R.~W. Jacobs, P.~Knappe, and P.~Wollan.
\newblock {A Menger-Type Theorem for Two Induced Paths}.
\newblock {\em SIAM J.\ Discrete Math.}, 38(2):1438--1450, 2024.

\bibitem{AJKW}
S.~Albrechtsen, R.~W. Jacobs, P.~Knappe, and P.~Wollan.
\newblock A characterisation of graphs quasi-isometric to ${K}_4$-minor-free graphs.
\newblock To appear in Combinatorica. arXiv:2408.15335.

\bibitem{BalMcMmet}
J.~Balig\'acs and J.~MacManus.
\newblock The metric {Menger} problem.
\newblock arXiv:2403.05630.

\bibitem{BBEGLPS}
M.~Bonamy, N.~Bousquet, L.~Esperet, C.~Groenland, C.-H. Liu, F.~Pirot, and A.~Scott.
\newblock Asymptotic {Dimension} of {Minor}-{Closed} {Families} and {Assouad}-{Nagata} {Dimension} of {Surfaces}.
\newblock {\em J.\ Eur.\ Math.\ Soc.}, 26(10):3739--3791, 2023.

\bibitem{CDNRV}
V.~Chepoi, F.~F. Dragan, I.~Newman, Y.~Rabinovich, and Y.~Vax\`es.
\newblock Constant {Approximation} {Algorithms} for {Embedding} {Graph} {Metrics} into {Trees} and {Outerplanar} {Graphs}.
\newblock {\em Discrete \& Computational Geometry}, 47(1):187--214, 2012.

\bibitem{DHIM}
J.~Davies, R.~Hickingbotham, F.~Illingworth, and R.~McCarty.
\newblock Fat minors cannot be thinned (by quasi-isometries).
\newblock arXiv:2405.09383.

\bibitem{Bibel}
R.~Diestel.
\newblock {\em Graph Theory \emph{(6th edition)}}.
\newblock Springer-Verlag, 2024.

\bibitem{FujPapCoa}
K.~Fujiwara and P.~Papasoglu.
\newblock A coarse-geometry characterization of cacti.
\newblock {arXiv:2305.08512}.

\bibitem{GeoPapMin}
A.~Georgakopoulos and P.~Papasoglu.
\newblock {Graph minors and metric spaces}.
\newblock {\em Combinatorica}, 45:33, 2025.

\bibitem{NgScSeAsyII}
T.~Nguyen, A.~Scott, and P.~Seymour.
\newblock {Asymptotic structure. II. Path-width and additive quasi-isometry}.
\newblock Preprint 2024.

\bibitem{NgScSeAsyIV}
T.~Nguyen, A.~Scott, and P.~Seymour.
\newblock {Asymptotic structure. IV. A counterexample to the weak coarse Menger conjecture}.
\newblock arXiv:2508.14332.

\bibitem{NgScSeCoa}
T.~Nguyen, A.~Scott, and P.~Seymour.
\newblock Coarse tree-width.
\newblock arXiv:2501.09839.

\bibitem{GMXIII}
N.~Roberston and P.~Seymour.
\newblock Graph minors. {XIII}. {T}he disjoint paths problem.
\newblock {\em J.~Combin.\ Theory (Series B)}, 63(1):65--110, 1995.

\end{thebibliography}
